\documentclass[11pt, leqno]{amsart}
\usepackage{amsmath, amssymb, latexsym}
\usepackage{cases}
\usepackage{mathrsfs}
\usepackage{enumerate}
\usepackage{color}
\usepackage[colorlinks=true, pdfstartview=FitV, linkcolor=blue,%
citecolor=blue, urlcolor=blue]{hyperref}
\usepackage[all, knot]{xy}

\xyoption{arc}

\numberwithin{equation}{section}

\newtheorem{theorem}{Theorem}[section]
\newtheorem{corollary}[theorem]{Corollary}
\newtheorem{lemma}[theorem]{Lemma}
\newtheorem{proposition}[theorem]{Proposition}

\theoremstyle{definition}
\newtheorem{definition}[theorem]{Definition}
\newtheorem{remark}[theorem]{Remark}

\newcommand{\Hom}{\operatorname{Hom}}

\newcommand{\F}{\mathbf{F}}
\newcommand{\Q}{\mathbf{Q}}

\newcommand{\Z}{\mathbf{Z}}
\newcommand{\N}{\mathbf{N}}

\newcommand{\g}{\mathfrak{g}}

\title[Global bases for quantum Borcherds-Bozec algebras]
{Global bases for quantum Borcherds-Bozec algebras}

\author[Zhaobing Fan]{Zhaobing Fan}
\address{Harbin Engineering University,
Harbin, China}
\email{fanzhaobing@hrbeu.edu.cn}
\thanks{ }

\author[Seok-Jin Kang]{Seok-Jin Kang}
\address{Korea Research Institute of Arts and Mathematics,
Asan-si, Chungcheongnam-do, 31551, Korea}
\email{soccerkang@hotmail.com}

\thanks{}

\author[Young Rock Kim]{Young Rock Kim${}^{*}$}
\address{Graduate School of Education, Hankuk University of Foreign Studies, Seoul, 02450,  Korea}
\email{rocky777@hufs.ac.kr} %
\thanks{${}^{*}$ Corresponding author. All authors contribute equally.}

\author[Bolun Tong]{Bolun Tong}
\address{Harbin Engineering University,
Harbin, China}
\email{tbl\_2019@hrbeu.edu.cn}

\address{}


\begin{document}

\begin{abstract}

We  provide a construction of  global bases for quantum Borcherds-Bozec algebras  and 
their integrable highest weight representations.

\end{abstract}

\maketitle

\setcounter{tocdepth}{1}
\tableofcontents

\section*{Introduction}

\vskip 2mm

The {\it quantum Borcherds-Bozec algebras} were introduced by Bozec \cite{Bozec2014b, Bozec2014c,BSV2016} in a natrual way  when he solved a question asked by Lusztig in \cite{Lus93}. More precisely, if we consider a quiver with loops, the Grothendieck group arising from Lusztig sheaves on  representation varieties  is generated by the elementary simple perverse sheaves
$F_i^{(n)}$ with all vertices $i$ and $n\in \N$.
Bozec proved an analogue of the Gabber-Kac theorem for the negative part $U_q^-(\g)$ of  a
quantum Borcherds-Bozec algebra under some restrictions of the inner product, and showed that the above Grothendieck group is isomorphic to $U_{q}^{-}(\g)$, which gives a construction of its {\it canonical basis}.

\vskip 2mm

The canonical basis theory was first introduced by Lusztig in the simply-laced case in \cite{Lus90}, due to his geometric construction of the negative parts of  quantum groups, and it has been generalized to
 symmetric Kac-Moody type in \cite{Lus91,Lusztig}. On the other hand, Kashiwara constructed the crystal bases and global bases for quantum groups associated with symmetrizable Kac-Moody algebras in an algebraic way \cite{Kas90,Kas91}. In \cite{GL1993}, Grojnowski and Lusztig proved that Kashiwara's  global bases coincide with Lustig's canonical bases. The crystal basis theory has become one of the most central themes in combinatorial and  geometric  representation theory of quantum groups  because it provides  us with a very powerful combinatorial tool to investigate the structure of  quantum groups and their integrable representations. In \cite{JKK2005},  Jeong, Kang and Kashiwara developed the crystal basis theory for quantum Borcherds algebras, which were introduced in \cite{Kang95}.  In \cite{KS06}, Kang and Schiffmann gave a construction of canonical basis for quantum Borcherds algebras and proved that, when all the diagonal entries of the corresponding  Borcherds-Cartan matrices are non-zero, the canonical bases  coincide with global bases.

\vskip 3mm

Bozec's crystal basis theory for quantum Borcherds-Bozec algbras is based on {\it primitive generators}
${\mathtt a}_{il}$, ${\mathtt b}_{il}$ $(i,l) \in I^{\infty}$, not on the {\it Chevalley generators} $e_{il}$, $f_{il}$.
The primitive generators have simpler commutation relations than Chevalley generators. Bozec defined  the Kashiwara operators using primitive generators and proved several crucial theorems which are important steps   for Kashiwara's grand-loop argument \cite[Lemma 3.33, Lemma 3.34]{Bozec2014c}. Moreover, using Lusztig's and Nakajima's quiver varieties, he also gave a
geometric construction of ${\mathcal B}(\infty)$, the crystal of the negative half $U_{q}^{-}(\g)$, and ${\mathcal B}(\lambda)$, the crystal of the integrable highest weight representation $V(\lambda)$, respectively.

\vskip 3mm

The main goal of this paper is to construct the {\it global bases} for quantum Borcherds-Bozec algebras and their integrable highest weight representations.
As is the case with Bozec's crystal basis theory, we are primarily interested in primitive generators
and we will give a new presentation of quantum Borcherds-Bozec algebras in terms of primitive generators.

\vskip 3mm

As the first step. we give an explicit description of the radical ${\mathscr R}$ of Lusztig's bilinear form. We could  take a direct generalization of the  approach in \cite[Chapter 7]{Lusztig}, but it is a very lengthy and messy calculation. Instead, we take a different approach using the co-multiplication (Lemma \ref{radical}), and prove higher order quantum Serre relations (Theorem \ref{r1}).
Bozec already showed $U_{q}(\g)$ is a Hopf algebra with Lusztig's bilinear form as Hopf pairing. Hence  by  \cite[Lemma 3.2]{SV1999}, we have only to impose Drinfeld relations on the generators of $U_{q}(\g)$. Since the primitive generators can be expressed as homogeneous polynomials in Chevalley generators (\cite{Bozec2014b,Bozec2014c}), we obtain  a new presentation of $U_q(\g)$ in terms of
primitive generators (Theorem \ref{alternative}).

\vskip 3mm

Then we set up the frame work that can be found in \cite{Kas91, JKK2005}. However, we still
need more preparations.
In the case of quantum Borcherds-Bozec algebras,
for each $I^{\text{im}}$,
there are infinitely many generators with higher degrees.
Thus, compared with quantum Borcherds algebras,
we need to take a much more complicated approach to the construction of global bases.
To overcome these difficulties, we introduce a very natural and much expanded notion of balaced triples corresponding to the compositions or partitions of each higher degree of primitive generators (Proposition \ref{M}, Corollary \ref{M1}). As can be expected, to prove our assertions, the imaginary indices with higher
degrees should be treated with special care. In particular, the isotropic case (i.e., when
$a_{ii}=0$) requires very subtle and delicate treatment.

\vskip 3mm

Now we can follow the steps give in \cite{Kas91, JKK2005} and prove the existence and uniqueness
of global bases (Theorem \ref{A}). Proposition \ref{M} and Corollary \ref{M1} play the crucial roles in the process.

\vskip 3mm

This paper is organized as follows. In  Section 1, we give  an explicit  description of the radical $\mathscr R$ of  the inner prduct $( \ , \ )_L$ via   higher order quantum Serre  relations in  quantum Borcherds-Bozec algebras.  In Section 2,  we  give a new presentation of quantum Borcherds-Bozec algebras in terms of primitive generators as an application of higher order quantum Serre  relations.
In Section 3, we review the crystal basis theory for quantum Borcherds-Bozec algebras and give  canonical  charaterizations of the crystal bases
$(\mathcal L(\infty), \mathcal B(\infty))$ and $(\mathcal L(\lambda), \mathcal B(\lambda))$, respectively. We also define the quantum Boson algebra $\mathscr B_q(\g)$ for an arbitrary  Borcherds-Cartan datum.  In Section 4, we define the $\mathbb A$-forms $U_{\mathbb A}(\g)$ of $U_{q}(\g)$  and $V(\lambda)^{\mathbb A}$ of $V(\lambda)$, respectively.  We prove that $U_{\mathbb A}(\g)$ has the triangular decomposition and both $U_{\mathbb A}(\g)$ and $V(\lambda)^{\mathbb A}$ are stable under the Kashiwara operators.
Section 5 is devoted to the proof of existence and uniqueness of global bases. We prove Proposition \ref{M} and Corollary \ref{M1}, which are key ingredients for the proof of our main goal.

\vskip 3mm

\noindent\textbf{Acknowledgements.}

Z. Fan was partially supported by the NSF of China grant 11671108, the NSF of Heilongjiang Province grant JQ2020A001, and the Fundamental Research Funds for the central universities. S.-J.  Kang was supported by  the NSF of China grant 11671108. Young Rock Kim was supported by the National Research Foundation of Korea (NRF) grant funded by the Korea government (MSIT) (No. 2021R1A2C1011467).

\vspace{12pt}

\section{Higher order quantum Serre relations}

\vskip 2mm

Let $I$ be an index set which can be countably infinite. An integer-valued matrix $A=(a_{ij})_{i,j \in I}$ is called an {\it
even symmetrizable Borcherds-Cartan matrix} if it satisfies the following conditions:
\begin{itemize}
\item[(i)] $a_{ii}=2, 0, -2, -4, ...$,

\item[(ii)] $a_{ij}\le 0$ for $i \neq j$,

\item[(iii)] there exists a diagonal matrix $D=\text{diag} (r_{i} \in \Z_{>0} \mid i \in I)$ such that $DA$ is symmetric.
\end{itemize}

\vskip 2mm

Let $I^{\text{re}}=\{i \in I \mid a_{ii}=2 \}$,
$I^{\text{im}}=\{i \in I \mid a_{ii} \le 0\}$ and
$I^{\text{iso}}=\{i \in I \mid a_{ii}=0 \}$.
The elements of $I^{\text{re}}$ (resp. $I^{\text{im}}$, $I^{\text{iso}}$) are called the {\it real indices}
(resp. {\it imaginary indices}, {\it isotropic indices}).

\vskip 3mm

A {\it Borcherds-Cartan datum} consists of :

\ \ (a) an even symmetrizable Borcherds-Cartan matrix $A=(a_{ij})_{i,j \in I}$,

\ \ (b) a free abelian group $P$, the {\it weight lattice},

\ \ (c) $\Pi=\{\alpha_{i} \in P  \mid i \in I \}$, the set of {\it simple roots},

\ \ (d) $P^{\vee} := \Hom(P, \Z)$, the {\it dual weight lattice},

\ \ (e) $\Pi^{\vee}=\{h_i \in P^{\vee} \mid i \in I \}$, the set of {\it simple coroots}

\vskip 2mm

\noindent satisfying the following conditions:

\vskip 1mm

\begin{itemize}

\item[(i)] $\langle h_i, \alpha_j \rangle = a_{ij}$ for all $i, j \in I$,

\item[(ii)] $\Pi$ is linearly independent over $\Q$,

\item[(iii)] for each $i \in I$, there exists an element $\Lambda_{i} \in P$, called the {\it fundamental weights},  such that $$\langle h_j , \Lambda_i
\rangle = \delta_{ij} \ \ \text{for all} \ i, j \in I.$$
\end{itemize}

\vskip 2mm

We denote by
$$P^{+}:=\{\lambda \in P \mid \langle h_i, \lambda \rangle \ge 0 \ \text{for all} \ i \in I \},$$
the set of {\it dominant integral weights}. The free abelian group $Q:= \bigoplus_{i \in I} \Z \, \alpha_i$ is called the {\it root lattice}.
Set $Q_{+}: = \sum_{i \in I} \Z_{\ge 0}\, \alpha_{i}$ and $Q_{-}: = -Q_{+}$.  For $\beta = \sum k_i \alpha_i \in Q\textbf{}_{+}$, we define its {\it height} to be $|\beta|:=\sum k_i$.

\vskip 3mm

Let ${\mathfrak h} := \Q \otimes_{\Z} P^{\vee}$ be the {\it Cartan subalgebra}.
There exists a non-degenerate symmetric bilinear
form $( \ , \ )$ on ${\mathfrak h}^{*}$ satisfying
$$(\alpha_{i}, \lambda) = r_{i} \langle h_{i}, \lambda \rangle \quad
\text{for all}  \ \lambda \in {\mathfrak h}^{*}.$$

\vskip 3mm

For each  $i \in I^{\text{re}}$, we deinfe the {\it simple reflection} $\omega_{i} \in {\mathfrak h}^{*}$ by
$$\omega_{i}(\lambda)= \lambda - \langle h_{i}, \lambda \rangle \,  \alpha_{i} \ \ \text{for} \ \lambda \in {\mathfrak h}^{*}.$$
The subgroup $W$ of $GL({\mathfrak h}^{*})$ generated by the simple reflections $\omega_{i}$ $(i \in I^{\text{re}})$ is called the {\it Weyl group} of the Borcherds-Cartan datum given above.  It is easy to check that $(\ , \ )$ is $W$-invariant.

\vskip 3mm

Let $I^{\infty}:= (I^{\text{re}} \times \{1\}) \cup (I^{\text{im}}
\times \Z_{>0})$. If $i \in I^{\text{re}}$, we often write $i$ for $(i,1)$. Let $q$ be an indeterminate and set
$$q_i=q^{r_i},\quad q_{(i)}=q^{\frac{(\alpha_i,\alpha_i)}{2}}.$$
%Note that $q_i= q_{(i)}$ if $i \in I^{\text {re}}$.
For each $i \in I^{\text {re}}$ and $n \in \Z_{\geq 0}$, we define
$$[n]_i=\frac{q_{i}^n-q_{i}^{-n}}{q_{i}-q_{i}^{-1}},\quad [n]_i!=\prod_{k=1}^n [k]_i,\quad
{\begin{bmatrix} n \\ k \end{bmatrix}}_i=\frac{[n]_i!}{[k]_i![n-k]_i!}.$$
%For any $i,j\in I$, we write $(i,j)$ for $(\alpha_i,\alpha_j)$.

\vskip 3mm

Let $\mathscr F=\Q(q)\left< f_{il} \mid (i,l) \in I^{\infty} \right>$ be the free associative algebra over $\Q(q)$ generated by the symbols $f_{il}$ for $(i,l)\in I^{\infty}$. 
By setting $\text{deg} f_{il}= -l\alpha_{i} $, $\mathscr F$ becomes a $Q_-$-graded algebra. 
For a homogeneous element $u$ in $\mathscr F$, we denote by $|u|$ the degree of $u$, and for any subset $A \subseteq Q_{-}$, set
${\mathscr F}_{A}=\{ x\in {\mathscr F} \mid |x| \in A \}$.

\vskip 3mm

We define a {\it twisted} multiplication on $\mathscr F\otimes\mathscr F$ by
$$(x_1\otimes x_2)(y_1\otimes y_2)=q^{-(|x_2|,|y_1|)}x_1y_1\otimes x_2y_2$$
for all homogeneous elements $x_1,x_2,y_1,y_2\in \mathscr F$, and equip $\mathscr F$ with a co-multiplication $\varrho$ defined by
$$\varrho(f_{il})=\sum_{m+n=l}q_{(i)}^{-mn}f_{im}\otimes f_{in} \ \ \text{for}  \ (i,l)\in I^{\infty}.$$
Here, by convention, $f_{i0}=1$ and $f_{il}=0$ for $l<0$.

\vskip 3mm

\begin{proposition}\cite{Bozec2014b,Bozec2014c} \
{\rm For any family $\nu=(\nu_{il})_{(i,l)\in I^{\infty}}$ of non-zero elements in $\Q(q)$, there exists a symmetric bilinear form $( \ , \ )_L :\mathscr F\times\mathscr F\rightarrow \Q(q)$ such that

\begin{itemize}
\item[(a)] $(x, y)_{L} =0$ if $|x| \neq |y|$,

\item[(b)] $(1,1)_{L} = 1$,

\item[(c)] $(f_{il}, f_{il})_{L} = \nu_{il}$ for all $(i,l) \in
I^{\infty}$,

\item[(d)] $(x, yz)_{L} = (\varrho(x), y \otimes z)_L$  for all $x,y,z
\in {\mathscr F}$.
\end{itemize}
Here, $(x_1\otimes x_2,y_1\otimes y_2)_L=(x_1,y_1)_L(x_2,y_2)_L$ for any $x_1,x_2,y_1,y_2\in {\mathscr F}$.}
\end{proposition}

\vskip 3mm

We denote by $\mathscr R$ the radical of $( \ , \ )_L$.

\vskip 3mm

Let $\mathcal C_l$ be the set of compositions $\mathbf c$ of $l$, and set $f_{i,\mathbf c}=f_{ic_1}\cdots f_{ic_m}$ for every  $i\in I^{\text{im}}$ and every $\mathbf c=(c_1,\cdots,c_m)\in \mathcal C_l$ .
Let $\mathcal C_l=\{\mathbf c_1, \mathbf c_2, \cdots,\mathbf c_r\}$. Then $f_{i,\mathbf c_1},\cdots,f_{i,\mathbf c_r}$ form a basis of $\mathscr F_{-l\alpha_i}$. Hence, for any homogeneous element $x$ in $\mathscr F$, $\varrho(x)$ can be written into the forms
$$\varrho(x)=x_{\mathbf c_1}\otimes f_{i,\mathbf c_1}+\cdots+x_{\mathbf c_r}\otimes f_{i,\mathbf c_r}+\ \text{terms of bidegree not in}\ Q_-\times -l\alpha_i,$$
$$\varrho(x)=f_{i,\mathbf c_1}\otimes x_{\mathbf c_1}'+\cdots+f_{i,\mathbf c_r}\otimes x_{\mathbf c_r}'+\ \text{terms of bidegree not in}\ -l\alpha_i\times Q_-.$$
We denote by $\varrho_{i,l}(x), \varrho^{i,l}(x):\mathscr F\rightarrow \mathscr F^r$ the $\Q(q)$-linear maps:
$$\varrho_{i,l}(x)=(x_{\mathbf c_1},\cdots,x_{\mathbf c_r}), \ \varrho^{i,l}(x)=(x_{\mathbf c_1}',\cdots,x_{\mathbf c_r}').$$

If $x,y$ are homogeneous elements such that $\varrho_{i,k}(y)=0$ for any $k>0$,  then we have
$$\varrho_{i,l}(xy)=q^{l(\alpha_i,|y|)}\varrho_{i,l}(x)y\ \ \text{and} \ \ \varrho_{i,l}(yx)=y\varrho_{i,l}(x).$$
\noindent
Here, $\varrho_{i,l}(x)y=(x_{\mathbf c_1}y,\cdots,x_{\mathbf c_r}y)$ and $y\varrho_{i,l}(x)=(yx_{\mathbf c_1},\cdots,yx_{\mathbf c_r})$ if $\varrho_{i,l}(x)=(x_{\mathbf c_1},\cdots,x_{\mathbf c_r})$.

\vskip 2mm

Similarly, if $\varrho^{i,k}(y)=0$ for any $k>0$, we have
$$\varrho^{i,l}(xy)=\varrho^{i,l}(x)y\ \ \text{and} \ \ \varrho^{i,l}(yx)=q^{l(\alpha_i,|y|)}y\varrho^{i,l}(x).$$

\vskip 3mm

For $i\in I^{\text{re}}$, we define the $\Q(q)$-linear maps $\varrho_i,\varrho^i:\mathscr F\rightarrow \mathscr F$ by
$$\varrho_i(1)=0, \varrho_i(f_{j,k})=\delta_{i,j}\delta_{k,1}, \ \text{and} \ \varrho_i(xy)=q^{(|y|,\alpha_i)}\varrho_i(x)y+x\varrho_i(y),$$
$$\varrho^i(1)=0, \varrho^i(f_{j,k})=\delta_{i,j}\delta_{k,1}, \ \text{and} \ \varrho^i(xy)=\varrho^i(x)y+q^{(|x|,\alpha_i)}x\varrho^i(y)$$
for all homogeneous elements $x,y$. Note that for any homogeneous element  $x\in \mathscr F$, we have
$$\varrho(x)=\varrho_i(x)\otimes f_{i}+\ \text{terms of other bi-homogeneities},$$
$$\varrho(x)=f_{i}\otimes \varrho^i(x)+\ \text{terms of other bi-homogeneities}.$$

\vskip 3mm

The following lemma can be derived directly from the definitions of $\varrho_{i,l}$ and $\varrho^{i,l}$.

\vskip 3mm

\begin{lemma}\label{radical}\
{\rm \begin{itemize}
\item [(a)] If $i\in I^{\text{re}}$, then for any $x,y \in \mathscr F$, we have
$$(y f_{i},x)_L=(f_i,f_i)_L(y,\varrho_i(x))_L, \ (f_{i}y,x)_L=(f_i,f_i)_L(y,\varrho^i(x))_L.$$
\item [(b)] If $i\in I^{\text{im}}$, let $x\in \mathscr F$ with $\varrho_{i,l}(x)=(x_{\mathbf c_1},\cdots,x_{\mathbf c_r})$ and $\varrho^{i,l}(x)=(x_{\mathbf c_1}',\cdots,x_{\mathbf c_r}')$. Then for any $x\in \mathscr F$, we have
$$(yf_{il},x)_L=(f_{il},f_{i,\mathbf c_1})_L(y,x_{\mathbf c_1})_L+\cdots+(f_{il},f_{i,\mathbf c_r})_L(y,x_{\mathbf c_r})_L,$$
$$(f_{il}y,x)_L=(f_{il},f_{i,\mathbf c_1})_L(y,x_{\mathbf c_1}')_L+\cdots+(f_{il},f_{i,\mathbf c_r})_L(y,x_{\mathbf c_r}')_L.$$
\item [(c)] Let $x\in \mathscr F$ be a homogeneous element with $|x|\neq 0$, we have
\begin{itemize}
\item [(i)] if $\varrho_{i,l}(x)\in \mathscr R$ for any $(i,l)\in I^{\infty}$, then $x\in\mathscr R$,
\item [(ii)] if $\varrho^{i,l}(x)\in \mathscr R$ for any $(i,l)\in I^{\infty}$, then $x\in\mathscr R$.
\end{itemize}Here, if $i\in I^{\text{im}}$, $\varrho^{i,l}(x)\in \mathscr R$ means each component of $\varrho^{i,l}(x)$ belongs to $\mathscr R$.
\end{itemize}}
\end{lemma}

\vskip 3mm

For any $i\in I^{\text{re}}$ and $n \in \N$, set
$$f_i^{(n)}=\frac{f_i^{n}}{[n]_i!}.$$
By a similar argument in \cite[1.4.2]{Lusztig}, we can prove:

\vskip 3mm
\begin{lemma}
{\rm We have
\begin{equation}
\varrho (f_i^{(n)})=\sum_{p+p'=n}q_i^{-pp'}f_i^{(p)}\otimes f_i^{(p')}
\end{equation}
for any $i\in I^{\text{re}}$ and $n \in \N$.}
\end{lemma}

\vskip 3mm

\begin{theorem}\label{r1}
{\rm Assume that $i\in I^{\text{re}}$, $j\in I$ and $i\neq j$. Let $m\in \Z_{>0}$, $n\in \Z_{\geq 0}$ with $m>-a_{ij}n$. Then for any $\mathbf c\in\mathcal C_n$, the following element of $\mathscr F$
\begin{equation}\label{qsr}
\mathtt F_{i,j,m,n,\mathbf c,\pm 1}=\sum_{r+s=m}(-1)^rq_i^{\pm r(-a_{ij}n-m+1)}f_i^{(r)}f_{j,\mathbf c}f_i^{(s)}
\end{equation}
belongs to $\mathscr R$. Here, we put $f_{j,\mathbf c}=f_j^{n}$ for $j \in I^{\text{re}}$.}
\begin{proof}
If $n=0$, then
$$\mathtt F_{i,j,m,0,{\mathbf c},\pm 1}=\sum_{r+s=m}(-1)^rq_i^{\pm r(1-m)}f_i^{(r)}f_i^{(s)}.$$
Since $\sum_{r+s=m}(-1)^rq_i^{\pm r(1-m)}\begin{bmatrix} m \\ r \end{bmatrix}_i=0$, we have $\mathtt F_{i,j,m,0,{\mathbf c},\pm 1}=0$.

We first assume that $j\in I^{\text{im}}$. For $0<k\leq n$ and $\mathbf c=(n_1,\cdots,n_t)\in \mathcal C_n$, we have
\begin{equation}
\begin{aligned}
& \varrho^{j,k}(f_i^{(r)}f_{j,\mathbf c}f_i^{(s)})= \varrho^{j,k}(f_i^{(r)}f_{j,\mathbf c})f_i^{(s)}=q^{-(r\alpha_i,k\alpha_j)}f_i^{(r)}\varrho^{j,k}(f_{j,\mathbf c})f_i^{(s)}\\
& \phantom{\varrho^{j,k}(f_i^{(r)}f_{j,\mathbf c}f_i^{(s)})}=q^{-(r\alpha_i,k\alpha_j)}f_i^{(r)}\left(\beta_{a_1,\cdots,a_t} f_{j,(n_1-a_1,\cdots,n_t-a_t)}\right)_{\substack{a_1\leq n_1,\cdots,a_t\leq n_t \\ a_1+\cdots+a_t=k}}f_i^{(s)},
\end{aligned}
\end{equation}
where $$\beta_{a_1,\cdots,a_t}=q_{(j)}^{\sum_{h=1}^ta_h(a_h-n_h)}q_{(j)}^{2\sum_{1\leq p< q\leq t}(a_p-n_p)a_q}.$$
Note that $m>-a_{i,j}n\geq-a_{i,j}(n-k)$ and
$$q_i^{ -r(-a_{ij}n-m+1)}q^{-(r\alpha_i,k\alpha_j)}=q_i^{ -r[-a_{ij}(n-k)-m+1]}.$$
Therefore each component of $\varrho^{j,k}(\mathtt F_{i,j,m,n,\mathbf c,-1})$ is a scalar multiple of $\mathtt F_{i,j,m,n-k,\mathbf c',-1}$ for some $\mathbf c'\in \mathcal C_{n-k}$.

Since $i\in I^{\text{re}}$, we have
\begin{equation}
\begin{aligned}
& \varrho^{i}(f_i^{(r)}f_{j,\mathbf c}f_i^{(s)})=\varrho^i(f_i^{(r)}f_{j,\mathbf c})f_i^{(s)}+q^{-(r\alpha_i+n\alpha_j,\alpha_i)}q_i^{1-s}f_i^{(r)}f_{j,\mathbf c}f_i^{(s-1)}\\
& \phantom{\varrho^{i}(f_i^{(r)}f_{j,\mathbf c}f_i^{(s)})}=q_i^{1-r}f_i^{(r-1)}f_{j,\mathbf c}f_i^{(s)}+q^{-(r\alpha_i+n\alpha_j,\alpha_i)}q_i^{1-s}f_i^{(r)}f_{j,\mathbf c}f_i^{(s-1)}.
\end{aligned}
\end{equation}
Hence
\begin{equation}
\begin{aligned}
& \varrho^{i}(\mathtt F_{i,j,m,n,\mathbf c,-1})=\sum_{r+s=m}(-1)^rq_i^{-r(-a_{ij}n-m+1)}q_i^{1-r}f_i^{(r-1)}f_{j,\mathbf c}f_i^{(s)}\\
& \phantom{\varrho^{i}(\mathtt F_{i,j,m,n,\mathbf c,-1})=}+\sum_{r+s=m}(-1)^rq_i^{-r(-a_{ij}n-m+1)}q^{-(r\alpha_i+n\alpha_j,\alpha_i)}q_i^{1-s}f_i^{(r)}f_{j,\mathbf c}f_i^{(s-1)}.
\end{aligned}
\end{equation}
Note that the coefficient of  $f_i^{(r)}f_{j,\mathbf c}f_i^{(s-1)}$ is
\begin{equation}
\begin{aligned}
& \ \ \ q_i^{-(r+1)(-a_{ij}n-m+1)}q_i^{-r}-q_i^{-r(-a_{ij}n-m+1)}q^{-(r\alpha_i+n\alpha_j,\alpha_i)}q_i^{1-m+r}\\
& =q_i^{-(r+1)(-a_{ij}n-m+1)}q_i^{-r}-q_i^{-r(-a_{ij}n-m+1)}q_i^{-2r-na_{ij}}q_i^{1-m+r}\\
%& =q_i^{(r+1)(-a_{ij}n-m+1)+1}(1-q_i^{2m+2na_{ij}-2})\\
& =q_i^{-r(-a_{ij}n-m+2)}q_i^{a_{ij}n+m-1}(1-q_i^{-2m-2na_{ij}+2}).
\end{aligned}
\end{equation}
Therefore
\begin{equation*}
\begin{aligned}
&\varrho^{i}(\mathtt F_{i,j,m,n,\mathbf c,-1})=(1-q_i^{-2m-2na_{ij}+2})q_i^{a_{ij}n+m-1}\cdot\sum_{r+s=m-1}(-1)^rq_i^{-r(-a_{ij}n-m+2)}f_i^{(r)}f_{j,\mathbf c}f_i^{(s)} \\
& \phantom{\varrho^{i}(\mathtt F_{i,j,m,n,\mathbf c,-1})}
\begin{aligned}
&=
\begin{cases}
\beta \mathtt F_{i,j,m-1,n,\mathbf c,-1}  & \text{if} \ \ m>-a_{ij}n+1,\\
 0 & \text{if} \ \ m=-a_{ij}n+1.
\end{cases}\\
\end{aligned}
\end{aligned}
\end{equation*}
Here $\beta=(1-q_i^{-2m-2na_{ij}+2})q_i^{a_{ij}n+m-1}$ is a constant.

By the induction and Lemma \ref{radical}(c), we get $\mathtt F_{i,j,m,n,\mathbf c, -1}$ belongs to $\mathscr R$ for $j\in I^{\text{im}}$. The case of $j\in I^{\text{re}}$ follows from the same way. In the meanwhile, one can show that $\mathtt F_{i,j,m,n,\mathbf c, +1}$ belongs to $\mathscr R$ by using the operators $\varrho_{j,k}$ and $\varrho_i$ in the above process.
\end{proof}
\end{theorem}

\vskip 3mm

In particular, when $m=1-la_{ij}, n=l$ and $\mathbf c=(l)$,  by Theorem \ref{r1},  we conclude
$$\mathtt F_{i,j,m,n,\mathbf c,\pm 1}=\begin{cases}\sum_{r+s=1-la_{ij}}(-1)^rf_i^{(r)}f_{j}^{(l)}f_i^{(s)} \quad\text{if} \ j\in I^{\text{re}},\\ \sum_{r+s=1-la_{ij}}(-1)^rf_i^{(r)}f_{jl}f_i^{(s)} \quad\text{if} \ j\in I^{\text{im}}\end{cases}$$
belongs to $\mathscr R$.

\vskip 3mm

\begin{lemma} \label{lem:r2}
{\rm
Let $(i,k),(j,l)\in I^{\infty}$ such that $a_{ij}=0$. Set $X=f_{ik}f_{jl}-f_{jl}f_{ik}$, Then $X\in \mathscr R$.
\begin{proof}
Note that if $i,j\in I^{\text{re}}$, then $X=f_{i}f_{j}-f_{j}f_{i}$. Since $i$ and $j$ cannot be equal, we have $X=-\mathtt F_{i,j,m=1,n=1}$.

\vskip 2mm
If $i\in I^{\text{re}}$ and $j\in I^{\text{im}}$, we have $X=f_{i}f_{jl}-f_{jl}f_{i}=-\mathtt F_{i,j,m=1,n=l,\mathbf c=(l)}$.

\vskip 2mm
We now assume that $i,j\in I^{\text{im}}$ and $i=j$; i.e.,  $i\in I^{\text{iso}}$. Note for any $0<s\leq k+l$, we have
\begin{equation*}
\begin{aligned}
& \varrho^{i,s}(X)=\varrho^{i,s}(f_{ik}f_{il}-f_{il}f_{ik})\\
& \phantom{\varrho^{i,s}(X)}=\left( f_{i,k-a_1}f_{i,l-a_2}-f_{i,l-a_2}f_{i,k-a_1}\right)_{\substack{a_1\leq k,a_2\leq l \\ a_1+a_2=s}}.
\end{aligned}
\end{equation*}
Thus  we can show  $X\in \mathscr R$ by induction.

\vskip 2mm
Finally, if $i,j\in I^{\text{im}}$ and $i\neq j$, then for any $0<s\leq k$ and $0<t\leq l$, we have
\begin{equation*}
\begin{aligned}
& \varrho^{i,s}(X)=\varrho^{i,s}(f_{ik}f_{jl}-f_{jl}f_{ik})=q_{(i)}^{-s(k-s)} (f_{i,k-s}f_{jl}-f_{jl}f_{i,k-s}),\\
& \varrho^{j,t}(X)=\varrho^{j,t}(f_{ik}f_{jl}-f_{jl}f_{ik})=q_{(j)}^{-t(l-t)} (f_{ik}f_{j,l-t}-f_{j,l-t}f_{ik}).
\end{aligned}
\end{equation*}
We can also show   $X\in \mathscr R$ by the induction.
\end{proof}
}\end{lemma}

\vskip 8mm

\section{Quantum Borcherds-Bozec algebras}

\vskip 2mm

From now on, we always assume that
\begin{equation} \label{eq:assumption}
\nu_{il} \in 1+q\Z_{\geq0}[[q]]\ \ \text{for all} \ (i,l)\in I^{\infty}.
\end{equation}
Under this assumption, the bilinear form $( \ , \ )_L$ is non-degenerate on $\mathscr F(i)=\bigoplus _{l\geq 1} {\mathscr F}_{-l \alpha_i}$ for $i\in I^{\rm{im}} \backslash I^{\rm{iso}}$. Moreover,  the two-side ideal $\mathscr R$ is generated by
$$ \sum_{r+s=1-la_{ij}}(-1)^rf_i^{(r)}f_{jl}f_i^{(s)} \ \ \text{for} \ i\in
I^{\text{re}},(j,l)\in I^{\infty} \ \text {and} \ i \neq (j,l),$$
and $f_{ik}f_{jl}-f_{jl}f_{ik}$ for all $(i,k),(j,l)\in I^{\infty}$ with $a_{ij}=0$ (cf. \cite[Proposition 14]{Bozec2014b}).

\vskip 3mm

Given a Borcherds-Cartan datum $(A, P, \Pi, P^{\vee}, \Pi^{\vee})$, we denote by $\widehat {U}$  the associative algebra over $\Q(q)$ with $\mathbf 1$,
generated by the elements $q^h$ $(h\in P^{\vee})$ and $e_{il},
f_{il}$ $((i,l) \in I^{\infty})$ with defining relations
\begin{equation} \label{eq:rels}
\begin{aligned}
& q^0=\mathbf 1,\quad q^hq^{h'}=q^{h+h'} \ \ \text{for} \ h,h' \in P^{\vee} \\
& q^h e_{jl}q^{-h} =  q^{l \, \langle h, \alpha_j \rangle} e_{jl}, \ \ q^h f_{jl}q^{-h}
=  q^{-l \, \langle h, \alpha_j  \rangle }  f_{jl}\ \ \text{for} \ h \in P^{\vee}, (j,l)\in I^{\infty}, \\
& \sum_{r+s=1-la_{ij}}(-1)^r
{e_i}^{(r)}e_{jl}e_i^{(s)}=0 \ \ \text{for} \ i\in
I^{\text{re}},(j,l)\in I^{\infty} \ \text {and} \ i \neq (j,l), \\
& \sum_{r+s=1-la_{ij}}(-1)^r
{f_i}^{(r)}f_{jl}f_i^{(s)}=0 \ \ \text{for} \ i\in
I^{\text{re}},(j,l)\in I^{\infty} \ \text {and} \ i \neq (j,l), \\
& e_{ik}e_{jl}-e_{jl}e_{ik} = f_{ik}f_{jl}-f_{jl}f_{ik} =0 \ \ \text{for} \ a_{ij}=0.
\end{aligned}
\end{equation}
We extend the grading by setting $|q^h|=0$ and $|e_{il}|= l \alpha_{i}$.

\vskip 3mm

The algebra $\widehat{U}$ is endowed with a co-multiplication
$\Delta: \widehat{U} \rightarrow \widehat{U} \otimes \widehat{U}$
given by
\begin{equation} \label{eq:comult}
\begin{aligned}
& \Delta(q^h) = q^h \otimes q^h, \\
& \Delta(e_{il}) = \sum_{m+n=l} q_{(i)}^{mn}e_{im}\otimes K_{i}^{-m}e_{in}, \\
& \Delta(f_{il}) = \sum_{m+n=l} q_{(i)}^{-mn}f_{im}K_{i}^{n}\otimes f_{in},
\end{aligned}
\end{equation}
where $K_i=q_i^{h_i}$ $(i \in I)$.

\vskip 3mm

Let $\widehat{U}^{\leq0}$ be the subalgebra of $\widehat{U}$ generated by $f_{il}$ and $q^h$ for all $(i,l) \in I^{\infty}$ and $h\in P^{\vee}$, and $\widehat{U}^+$ be the subalgebra
generated by $e_{il}$ for all $(i,l) \in I^{\infty}$. We extend $( \ , \ )_L$ to a symmetric bilinear form $( \ , \ )_L$ on $\widehat{U}^{\leq 0}$ and on $\widehat{U}^+$ by setting
\begin{equation}
\begin{aligned}
& (q^h,1)_L=1,\ (q^h,f_{il})_L=0, \\
& (q^h,K_j)_L=q^{-\langle h, \alpha_j \rangle},\\
& (x,y)_L=(\omega(x),\omega(y))_L \ \ \text{for all} \ x,y\in \widehat{U}^+,
\end{aligned}
\end{equation}
where $\omega:\widehat{U}\rightarrow\widehat{U}$  is the involution defined by
$$\omega(q^h)=q^{-h},\ \omega(e_{il})=f_{il},\ \omega(f_{il})=e_{il}\ \ \text{for}\ h \in P^{\vee},\ (i,l)\in I^{\infty}.$$

\vskip 3mm

For any $x\in \widehat{U}$, we shall use the Sweedler's notation, and write
$$\Delta(x)=\sum x_{(1)}\otimes x_{(2)}.$$

\begin{definition} Following the Drinfeld double process, we define the {\it quantum Borcherds-Bozec algebra} $U_q(\g)$ associated with a given Borcherds-Cartan datum  $(A, P, P^{\vee}, \Pi,  \Pi^{\vee})$ as the quotient of $\widehat{U}$ by the relations
\begin{equation}\label{drinfeld}
\sum(a_{(1)},b_{(2)})_L\omega(b_{(1)})a_{(2)}=\sum(a_{(2)},b_{(1)})_La_{(1)}\omega(b_{(2)})\ \ \text{for all}\ a,b \in \widehat{U}^{\leq0}.
\end{equation}
\end{definition}
\vskip 3mm

\begin{remark} \label{rmk:SV}
The quantum Borcherds-Bozec algebra $U_q(\g)$ is a Hopf algebra constructed by the Drinfeld double process. By \cite[Lemma 3.2]{SV1999}, we have only to impose the commutation relations \eqref{drinfeld}  on the generators of $U_q(\g)$.
\end{remark}

\vskip 3mm

Let $U^+_q(\g)$ (resp. $U^-_q(\g)$) be the subalgebra of $U_q(\g)$ generated by $e_{il}$ (resp. $f_{il}$) for $(i,l)\in I^{\infty}$,
and $U^{0}_q(\g)$ the subalgebra of $U_q(\g)$ generated by $q^h$ for $h\in P^{\vee}$. We shall denote by $U$ (resp. $U^+$, $U^0$ and $U^-$) for $U_q(\g)$ (resp. $U^+_q(\g)$, $U_{q}^0(\g)$ and $U^-_q(\g)$) for simplicity. Then $U$ has the  {\it triangular decomposition} \cite{KK19}
$$U\cong U^-\otimes U^0 \otimes U^+.$$

\vskip 3mm

\begin{proposition}\cite{Bozec2014b,Bozec2014c}\label{prim}
{\rm For any $i\in I^{\text {im}}$ and $l\geq 1$, there exist unique elements $\mathtt b_{il}\in  U^-_{-l \alpha_{i}}$ and $\mathtt a_{il}=\omega (\mathtt b_{il})$ such that
\begin{itemize}\label{bozec}
\item[(1)] $\Q (q) \left<f_{il} \mid l\geq 1\right>=\Q (q) \left<\mathtt b_{il} \mid l\geq 1\right>$ and $\Q (q) \left<e_{il} \mid l\geq 1\right>=\Q (q) \left<\mathtt a_{il} \mid l\geq 1\right>$,
\item[(2)] $(\mathtt b_{il},z)_L=0$ for all $z\in \Q (q) \left<f_{i1} ,\cdots,f_{il-1}\right>$,\\
           $(\mathtt a_{il},z)_L=0$ for all $z\in \Q (q) \left<e_{i1} ,\cdots,e_{il-1}\right>$,
\item[(3)]  $\mathtt b_{il}-f_{il}\in \Q (q) \left<f_{ik} \mid k<l \right>$ and $\mathtt a_{il}-e_{il}\in \Q (q) \left<e_{ik} \mid k<l \right>$,
\item[(4)] $\overline{\mathtt b}_{il}=\mathtt b_{il},\ \overline{\mathtt a}_{il}=\mathtt a_{il}$,
\item[(5)] $\varrho(\mathtt b_{il})=\mathtt b_{il}\otimes 1+ 1\otimes \mathtt b_{il}, \ \varrho(\mathtt a_{il})=\mathtt a_{il}\otimes 1+1\otimes \mathtt a_{il}$,
\item[(6)] $\Delta(\mathtt b_{il})=\mathtt b_{il}\otimes 1+K_i^l\otimes \mathtt b_{il}, \ \Delta(\mathtt a_{il})=\mathtt a_{il}\otimes K_i^{-l}+ 1\otimes \mathtt a_{il}$,
\item[(7)] $S(\mathtt b_{il})=-K_i^{-l}\mathtt b_{il}, \ S(\mathtt a_{il})=-\mathtt a_{il}K_i^{l}$.
\end{itemize}
Here, $S$ is the antipode of $U$, and $^-:U^{\pm}\rightarrow U^{\pm}$ is the $\Q$-algebra involution defined by
$
\overline{e}_{il}=e_{il},\ \overline{f}_{il}=f_{il}$ and $\overline{q}=q^{-1}$.}
\end{proposition}

The elements $\mathtt a_{il}$, $\mathtt b_{il}$ $((i,l) \in I^{\infty})$  are called the {\it primitive generators} of the  quantum Borcherds-Bozec algebra $U_{q}(\g)$.

\vskip 3mm

By setitng $\tau_{il}=(\mathtt a_{il},\mathtt a_{il})_L=(\mathtt b_{il},\mathtt b_{il})_L$, we get the following commutation relations in $U_q(\g)$
\begin{equation}\label{news}
\mathtt a_{il}\mathtt b_{jk}-\mathtt b_{jk}\mathtt a_{il}=\delta_{ij}\delta_{lk}\tau_{il}(K_i^{l}-K_i^{-l}).
\end{equation}

\vskip 3mm
Let $\mathcal C_{l}$ (resp. $\mathcal P_{l}$) be the set of compositions (resp. partitions) of $l$. For $i\in I^{\text{im}}$, we define
\begin{equation*}
\mathscr C_{i,l}=\begin{cases} \mathcal C_{l}& \text{if} \ i\in I^{\text{im}}\backslash I^{\text{iso}}, \\ \mathcal P_{l}& \text{if} \ i\in I^{\text{iso}} \end{cases}
\end{equation*}
and $\mathscr C_i=\bigsqcup_{l\geq 0}\mathscr C_{i,l}$. For $i\in I^{\text{re}}$, we just put $\mathscr C_{i,l}=\{l\}$.
\vskip 3mm

  Assume that $i\in I^{\text{im}}$.  Let  $\mathbf c =(c_1,\cdots,c_t)\in \mathscr C_{i,l}$  and set
  $$ \mathtt b_{i,\mathbf c}=\mathtt b_{ic_1}\cdots\mathtt b_{ic_t} , \ \mathtt a_{i,\mathbf c}=\mathtt a_{ic_1}\cdots\mathtt a_{ic_t} \ \text{and} \ \tau_{i,\mathbf c}=\tau_{ic_1}\cdots\tau_{ic_t} .$$
 Note that $\{ \mathtt b_{i,\mathbf c} \mid \mathbf c \in \mathscr C_{i,l}\}$ forms a basis of  $U^-_{-l \alpha_{i}}$. For each $i\in I^{\text {re}}$, we put $\mathtt b_{i1}=f_{i1}$, $\mathtt a_{i1}=e_{i1}$ and $\tau_i=\nu_i$.
Sometimes, we simply write $\mathtt a_i$ (resp. $\mathtt b_i$) in this case.
\vskip 3mm
%\begin{remark}
%We use partitions when $i\in I^{\text{iso}}$ and use compositions when $i\in I^{\text{im}}\backslash I^{\text{iso}}.$
%\end{remark}

%There exists another set of generators in $U_q(\g)$ called {\it primitive generators}. They satisfy a simpler set of commutation relations.

\vskip 3mm

\begin{remark} \label{lem:prop3.3} \hfill

\vskip 2mm

{\rm
(1) Each $\lambda\in\mathcal P_l$ can be written as the form $\lambda=1^{\lambda_1}2^{\lambda_2}\cdots l^{\lambda_l}$, where $\lambda_k$ are non-negative integers such that $\lambda_1+2\lambda_2+\cdots+l\lambda_l=l$. For $i\in I^{\text{iso}}$, we have
$$\mathtt b_{il}=f_{il}-\sum_{\lambda\in \mathcal P_l\backslash (l)} \frac{1}{\prod_{k=1}^{l} \lambda_k!} \mathtt b_{i,\lambda}.$$

Note that assumption (\ref{eq:assumption}) implies $\nu_{il}\equiv 1 \ (\text{mod} \ q)$. Hence we have
$\tau_{il}\equiv \frac{1}{l} \ (\text{mod} \ q)$
by the following equation
$$\sum_{\lambda\in \mathcal P_l} \frac{1}{\prod_{k=1}^{l} k^{\lambda_k}\lambda_k!}=1.$$

(2) Under the assumption (\ref{eq:assumption}), if $i\in I^{\text{im}}\backslash I^{\text{iso}}$, it was shown in \cite[Lemma 3.32]{Bozec2014c} that $\tau_{il}\equiv 1 \ (\text{mod} \ q)$
for all $l\geq 1$.

\vskip 2mm

}\end{remark}

\vskip 3mm

We now give an alternative presentation  of the quantum Borcherds-Bozec algebra $U_q(\g)$.

\begin{theorem} \label{alternative}
{\rm
The quantum Borcherds-Bozec algebra $U_{q}(\g)$ is generated by the primitive generators
$\mathtt a_{il}, \mathtt b_{il}$ $((i,l) \in I^{\infty})$ and $q^h$ $(h\in P^{\vee})$ subject to the defining relations
\begin{equation} \label{eq:newrels}
\begin{aligned}
& q^0=\mathbf 1,\quad q^hq^{h'}=q^{h+h'} \ \ \text{for} \ h,h' \in P^{\vee} \\
& q^h \mathtt a_{jl}q^{-h} = q^{l \langle h, \alpha_j \rangle} \mathtt a_{jl}, \ \ q^h \mathtt b_{jl}q^{-h} = q^{-l \langle h, \alpha_j \rangle } \mathtt b_{jl}\ \ \text{for} \ h \in P^{\vee}, (j,l)\in I^{\infty}, \\
& \mathtt a_{il}\mathtt b_{jk}-\mathtt b_{jk}\mathtt a_{il}=\delta_{ij}\delta_{lk}\tau_{il}(K_i^{l}-K_i^{-l}), \\
& \sum_{r+s=1-la_{ij}}(-1)^r
{\mathtt a_i}^{(r)}\mathtt a_{jl}\mathtt a_i^{(s)}=0 \ \ \text{for} \ i\in
I^{\text{re}},(j,l)\in I^{\infty} \ \text {and} \ i \neq (j,l), \\
& \sum_{r+s=1-la_{ij}}(-1)^r
{\mathtt b_i}^{(r)}\mathtt b_{jl}\mathtt b_i^{(s)}=0 \ \ \text{for} \ i\in
I^{\text{re}},(j,l)\in I^{\infty} \ \text {and} \ i \neq (j,l), \\
& \mathtt a_{ik}\mathtt a_{jl}-\mathtt a_{jl}\mathtt a_{ik} = \mathtt b_{ik}\mathtt b_{jl}-\mathtt b_{jl}\mathtt b_{ik} =0 \ \ \text{for} \ a_{ij}=0.
\end{aligned}
\end{equation}
}
\end{theorem}

\begin{proof}

As is the case with ${\mathtt a}_{il}$ and ${\mathtt b}_{il}$.
For a composition or partition $\mathbf{c} = (c_{1}, \ldots, c_{t})$, we write
$e_{i, \mathbf{c}} = e_{ic_1} \cdots e_{ic_t}$ and
$f_{i, \mathbf{c}} = f_{ic_1} \cdots f_{ic_t}$.

\vskip 2mm

Recall that each ${\mathtt a}_{jl}$ (resp. ${\mathtt b}_{jl}$) can be written as a homogeneous polynomial in
$e_{jk}$'s (resp. $f_{jk}$'s) for $1 \le k \le l$.  Thus we may write
$${\mathtt a}_{jl} = \sum_{\mathbf c} \alpha_{\mathbf c} \,  e_{j, \mathbf c}, \quad
{\mathtt b}_{jl} = \sum_{{\mathbf c}'} \beta_{{\mathbf c}'} \, f_{j,{ \mathbf c}'}, $$
where $\mathbf c$ and ${\mathbf c}'$ are compositions (or partitions) of $l$.

Then we get
\begin{equation*}
\begin{aligned}
q^{h} \, {\mathtt a}_{jl} \, q^{-h} &= q^{h} \, (\sum_{\mathbf c} \, \alpha_{\mathbf c} \, e_{j, \mathbf {c}} )q^{-h}
 =  \sum_{\mathbf c} \, \alpha_{\mathbf c} \, q^{h} e_{j, \mathbf{c}} q^{-h} \\
&  = \sum_{\mathbf c} \, \alpha_{\mathbf c} \, q^{h} \, e_{jc_1} e_{jc_2} \cdots e_{jc_t}\, q^{-h} \\
& =  \sum_{\mathbf c} \, \alpha_{\mathbf c} \ q^{c_1 \, \langle h, \alpha_j \rangle} \, e_{jc_{1}}
\,q^{c_2 \, \langle h, \alpha_j \rangle} \, e_{jc_{2}}  \cdots q^{c_t \, \langle h, \alpha_j \rangle} e_{jc_{t}} \\
& = q^{l \langle h, \alpha_j \rangle}  \sum_{\mathbf c} \, \alpha_{\mathbf c} \, e_{j, \mathbf {c}}
=  q^{l \langle h, \alpha_j \rangle} {\mathtt  a}_{jl}.
\end{aligned}
\end{equation*}

Similarly, we can show $q^{h} {\mathtt b}_{jl} q^{-h} = q^{- l \langle h, \alpha_j \rangle} \, {\mathtt b}_{jl}$.

\vskip 3mm

Since $U_{q}(\g)$ is a Hopf algebra, by \cite[Lemma 3.2]{SV1999}, we have only to impose the relations \eqref{drinfeld} on the generators, which gives  the relations in the third line of \eqref{eq:newrels}.

\vskip 3mm

For the quantum Serre relations, for $i \in I^{\text{re}}$, $i \neq (j,l)$,  by Theorem \ref{r1}, we have
\begin{equation*}
\begin{aligned}
& \sum_{r + s = 1- l a_{ij}} (-1)^{r} {\mathtt a}_{i}^{(r)} \, {\mathtt a}_{jl} \, {\mathtt a}_{i}^{(s)} \\
& = \sum_{r + s = 1 - l a_{ij}} (-1)^{r} e_{i}^{(r)} \, \left(\sum_{\mathbf{c}} \alpha_{\mathbf c} e_{j, \mathbf{c}} \right)
\, e_{i}^{(s)} \\
& = \sum_{\mathbf{c}} \alpha_{\mathbf{c}} \left(\sum_{r + s = 1 - l a_{ij}} (-1)^{r} e_{i}^{(r)}
e_{j, \mathbf{c}} e_{i}^{(s)} \right)= 0.
\end{aligned}
\end{equation*}

\vskip 2mm

The other relations can be verified in a similar manner.
\end{proof}

\vskip 8mm

\section{Crystal bases and polarization}

\vskip 2mm

\begin{definition}
For $i\in I^{\text{im}}$, we define the linear maps $\delta_{i,\mathbf c}, \delta^{i,\mathbf c}:U^{-}\rightarrow U^{-}$ by
$$\varrho(x)=\sum_{\mathbf c\in \mathscr C_{i}}\delta_{i,\mathbf c}\otimes \mathtt b_{i,\mathbf c}+\ \text{terms of bidegree not in}\ Q_-\times -\N \alpha_i,$$
$$\varrho(x)=\sum_{\mathbf c\in \mathscr C_{i}}\mathtt b_{i,\mathbf c}\otimes\delta^{i,\mathbf c}+\ \text{terms of bidegree not in}\ -\N \alpha_i\times Q_-.$$
\end{definition}

\vskip 3mm

Let $i\in I^{\text{im}}$, $l>0$. For any homogeneous elements  $x,y,z\in U^-$ and $\mathbf c=(c_1,\cdots,c_t)\in \mathscr C_i$, we have the following equations
\begin{equation}
\delta^{i,l}(xy)=\delta^{i,l}(x)y+q^{l(\alpha_i,|x|)}x\delta^{i,l}(y),
\end{equation}
\begin{equation}
\delta^{i,l}(\mathtt b_{i,\mathbf c})=\sum_{k:c_k=l}q_{(i)}^{-2l\sum_{j<k}c_j}\mathtt b_{i,\mathbf c\backslash c_k},
\end{equation}
\begin{equation}\label{c}
 \left[\mathtt a_{il},z\right]=\tau_{il}\left(\delta_{i,l}(z)K_i^l-K_i^{-l}\delta^{i,l}(z)\right),
 \end{equation}
where $\mathbf c\backslash c_k=(c_1,\cdots,\widehat{c}_k,\cdots,c_r)$ means removing $c_k$ from $\mathbf c$. From now on, we will denote the operator $\delta^{i,l}$ by $e'_{i,l}$.

\vskip 3mm

In \cite{Bozec2014c}, Bozec showed that  every  $u\in U^-$ can be written uniquely as
$$u=\sum_{\mathbf c\in \mathscr C_i}\mathtt b_{i,\mathbf c}u_{\mathbf c},$$
where $e'_{i,l}\, u_{\mathbf{c}} = 0$ for all $l\ge 1$ and $\mathbf c \in \mathscr C_i$. Moreover, if $u$ is homogenous, then every $u_{\mathbf c}$ is homogenous. Then the Kashiwara operators are defined by
\begin{equation*}
\begin{aligned}
& \widetilde e_{il}u=\begin{cases}\displaystyle\sum_{\mathbf c:c_1=l}\mathtt b_{i,\mathbf c\backslash c_1}u_{\mathbf c}\qquad
\quad \ \text{if}\ i\notin I^{\text{iso}},\\
\displaystyle\sum_{\mathbf c\in \mathscr C_i}\sqrt{\frac{m_l(\mathbf c)}{l}}\mathtt b_{i,\mathbf c\backslash l}u_{\mathbf c}\ \ \text{if}\ i\in I^{\text{iso}},\end{cases}\\
& \widetilde f_{il}u=\begin{cases}\displaystyle\sum_{\mathbf c\in \mathscr C_i}\mathtt b_{i,\mathbf (l,c)}u_{\mathbf c}\qquad\qquad \quad \ \ \text{if}\ i\notin I^{\text{iso}},\\
\displaystyle\sum_{\mathbf c\in \mathscr C_i}\sqrt{\frac{l}{m_l(\mathbf c)+1}}\mathtt b_{i,\mathbf c\cup l}u_{\mathbf c}\ \ \text{if}\ i\in I^{\text{iso}},\end{cases}
\end{aligned}
\end{equation*}
where $m_{l}(\mathbf{c}) = \# \{k \mid c_k = l\}$.

\vskip 3mm
\begin{remark}
Note that the square roots appear in the above definition. So we need to consider an extension $\F$ of $\Q$ that contains all the necessary square roots (see \cite[Remark 3.12]{Bozec2014c}).
\end{remark}

Let $\mathbb A_0=\{f\in \F(q) \mid f \ \text{is regular at} \ q=0\}$, and let
$\mathcal L(\infty)$ be the $\mathbb A_0$-submodule of $U^-$ spanned by the elements
$\widetilde{f}_{i_1,l_1}\cdots\widetilde{f}_{i_r,l_r}  \mathbf{1}$ for $r\geq 0$ and $(i_k,l_k)\in I^{\infty}$, where the Kashiwara operators $\widetilde{f}_i$ for $i\in I^{\text{re}}$ have been defined in \cite{Kas91}. Set
$$\mathcal B(\infty)=\{\widetilde{f}_{i_1,l_1}\cdots\widetilde{f}_{i_r,l_r} \mathbf{1}  \  \text{mod} \ q\mathcal L(\infty) \mid r \ge 0, (i_k, l_k) \in I^{\infty}  \}\subseteq
{\mathcal L}(\infty)/ q{\mathcal L}(\infty).$$
Then $(\mathcal L(\infty),\mathcal B(\infty))$ is a crystal basis of $U^-$.

\vskip 3mm

By \cite[Lemma 3.33]{Bozec2014c},  \cite[Proposition 5.1.2]{Kas91}, \cite[Lemma 7.39]{JKK2005}
and Remark \ref{lem:prop3.3},  we obtain:

\vskip 3mm

\begin{proposition}\label{P2}\

{\rm\begin{itemize}
\item [(i)] $(\mathcal L(\infty),\mathcal L(\infty))_L\subseteq \mathbb A_0.$

\item [(ii] $\mathcal L(\infty)=\{u\in U^- \mid  (u,\mathcal L(\infty))_L\subseteq\mathbb A_0\}=\{u\in U^- \mid  (u,u)_L\in\mathbb A_0\}$.
\end{itemize}

Let $( \ , \ )_L^0$ denote the $\F$-valued inner product on $\mathcal L(\infty)/ q\mathcal  L(\infty)$
obtained from $(\ , \ )_{L}$ on $\mathcal L(\infty)$ by setting $q=0$.
\begin{itemize}
\item [(iii)] $(\widetilde e_{il}u,v)_L^0=(u,\widetilde f_{il}v)_L^0$ for $u,v\in \mathcal L(\infty)/ q\mathcal  L(\infty)$ and $(i,l)\in I^{\infty}$.

\item [(iv)] $\mathcal B(\infty)$ is an orthonormal basis with respect to $( \ , \ )_L^0$. In particular, $( \ , \ )_L^0$ is positive definite.

\end{itemize}
}\end{proposition}

\vskip 3mm

Let $\lambda\in P^+$ and let $V(\lambda)$ be the irreducible highest weight $U_q(\g)$-module with highest weight $\lambda$ and highest weight vector $v_{\lambda}$. Then we have  a $U^-_q(\g)$-module isomorphism (cf. \cite{BSV2016}, \cite{KK19})
\begin{equation}\label{V}
\begin{aligned}
 V(\lambda)\simeq U^-_q(\g) \bigg/ %\sum_{(i,l)\in I^{\infty}}U_q(\g)e_{il}+\sum_{h\in P^{\vee}} U_q(\g)(q^h-q^{\lambda(h)})+
(\sum_{i\in I^{\text{re}}}U^-_q(\g)f_i^{\lambda(h_i)+1}+
\sum_{\substack{i\in I^{\text{im}}, \lambda(h_i)=0\\(i,l)\in I^{\infty}}}U^-_q(\g)f_{il} ).
\end{aligned}
\end{equation}

\vskip 3mm

Let $i\in I^{\text{im}}$ and $\lambda\in P^+$. In \cite{Bozec2014c}, Bozec showed that every $v\in V(\lambda)_\mu$ has a decomposition of the following form
$$v=\sum_{\mathbf c\in \mathscr C_i}\mathtt b_{i,\mathbf c}v_{\mathbf c},$$
where $v_{\mathbf c}\in V(\lambda)_{\mu+|\mathbf c|\alpha_i}$ and $e_{il}v_{\mathbf c}=0$ for all $l\geq 1$ and $\mathbf c\in \mathscr C_i$. Moreover, if we omit the terms $\mathtt b_{i,\mathbf c}v_{\mathbf c}$ with $|\mathbf c|\neq 0$ and $(\mu +|\mathbf c|\alpha_i,\alpha_i)=0$, which are equal to zero trivially, then the decomposition of $v$ is unique.

\vskip 3mm

Define the Kashiwara operators on $V(\lambda)$ by
\begin{equation*}
\begin{aligned}
& \widetilde e_{il}v=\begin{cases}\displaystyle\sum_{\mathbf c:c_1=l}\mathtt b_{i,\mathbf c\backslash c_1}v_{\mathbf c}& \text{if}\ i\notin I^{\text{iso}},\\
\displaystyle\sum_{\mathbf c\in \mathscr C_i}\sqrt{\frac{m_l(\mathbf c)}{l}}\mathtt b_{i,\mathbf c\backslash l}v_{\mathbf c}& \text{if}\ i\in I^{\text{iso}},\end{cases}\\
& \widetilde f_{il}v=\begin{cases}\displaystyle\sum_{\mathbf c\in\mathscr C_i}\mathtt b_{i,\mathbf (l,c)}v_{\mathbf c}& \text{if}\ i\notin I^{\text{iso}},\\
\displaystyle\sum_{\mathbf c\in \mathscr C_i}\sqrt{\frac{l}{m_l(\mathbf c)+1}}\mathtt b_{i,\mathbf c\cup l}v_{\mathbf c}& \text{if}\ i\in I^{\text{iso}}.\end{cases}
\end{aligned}
\end{equation*}

\vskip 3mm

Let $\mathcal L(\lambda)=\sum_{\iota_1,\cdots,\iota_s\in I^{\infty}}\mathbb A_0\widetilde{f}_{\iota_1}\cdots\widetilde{f}_{\iota_s}v_{\lambda}$ be an $\mathbb A_0$-submodule of $V(\lambda)$ and let
$$\mathcal B(\lambda)=\{\widetilde{f}_{\iota_1}\cdots\widetilde{f}_{\iota_s}v_{\lambda}\mid \iota_k\in I^{\infty}\}\backslash \{0\}\subseteq {\mathcal L}(\lambda)/ q{\mathcal L}(\lambda){\color{red}.}$$
Then  $(\mathcal L(\lambda),\mathcal B(\lambda))$ is a crystal basis of $V(\lambda)$.

\vskip 3mm

There exists a unique symmetric bilinear form $\{ -, - \}$ on $V(\lambda)$   such that
\begin{equation*}
\begin{aligned}
& \{v_{\lambda},v_{\lambda}\}=1, \\
& \{q^hv,v'\}=\{v,q^hv'\},\\
& \{\mathtt b_{il}v,v'\}=-\{v, K_i^{l}\mathtt a_{il}v'\} \ \ \text{if} \ i\in I^{\text{im}},\\
& \{\mathtt b_{i}v,v'\}=\frac{1}{q_i^{2}-1}\{v, K_i\mathtt a_{i}v'\} \ \ \text{if} \ i\in I^{\text{re}}\\
\end{aligned}
\end{equation*}
for every $v,v'\in V(\lambda)$ and $(i,l)\in I^{\infty}$.

\vskip 3mm

By \cite[Lemma 3.34]{Bozec2014c},  \cite[Proposition 5.1.1]{Kas91}, \cite[Corollary 7.37]{JKK2005} 
and Remark \ref{lem:prop3.3}, we obtain:

\vskip 3mm

\begin{proposition}\label{P3}\

{\rm\begin{itemize}
\item [(i)] $\{\mathcal L(\lambda),\mathcal L(\lambda)\}\subseteq \mathbb A_0.$
\item [(ivi] $\mathcal L(\lambda)=\{v\in V(\lambda) \mid  \{v,\mathcal L(\lambda)\}\subseteq\mathbb A_0\}=\{v\in V(\lambda) \mid  \{v,v\}\in\mathbb A_0\}$.
\end{itemize}

Let $\{ \ , \ \}_0$ denote the $\F$-valued inner product on $\mathcal L(\lambda)/ q\mathcal  L(\lambda)$ induced by $\{ \ , \ \}$ on $\mathcal L(\lambda)$.
\begin{itemize}
\item [(iii)] $\{\widetilde e_{il}u,v\}_0=\{u,\widetilde f_{il}v\}_0$ for $u,v\in \mathcal L(\lambda)/ q\mathcal  L(\lambda)$ and $(i,l)\in I^{\infty}$.
\item [(iv)] $\mathcal B(\lambda)$ is an orthonormal basis with respect to $\{ \ , \ \}_0$. In particular, $\{ \ , \ \}_0$ is positive definite.
\end{itemize}
}\end{proposition}

\vskip 3mm

The following proposition follows from Kashiwara's grand-loop argument, which describes the relations between $\mathcal B(\infty)$ and $\mathcal B(\lambda)$.

\begin{proposition}\label{P4}{\rm
Let $\pi_\lambda:U^-_q(\g)\rightarrow V(\lambda)$ be the $U^-_q(\g)$-module homomorphism given by $P\mapsto Pv_{\lambda}$. Then we have:
\begin{itemize}
\item [(i)] $\pi_\lambda(\mathcal L(\infty))=\mathcal L(\lambda)$; hence $\pi_{\lambda}$ induces the surjective homomorphism $$\overline \pi_{\lambda}:\mathcal L(\infty)/ q\mathcal  L(\infty)\rightarrow\mathcal L(\lambda)/ q\mathcal  L(\lambda).$$
\item [(ii)] $\{b\in \mathcal B(\infty)\mid \overline \pi_{\lambda}(b)\neq 0 \}$ is isomorphic to $\mathcal B(\lambda)$ under the map $\overline \pi_{\lambda}$.
\item [(iii)] If $b\in \mathcal B(\infty)$ satisfies $\overline \pi_{\lambda}(b)\neq 0$, then $\widetilde e_{il}\overline \pi_{\lambda}(b)=\overline \pi_{\lambda}(\widetilde e_{il}b)$.
\item [(iv)] $\widetilde f_{il}\circ\overline \pi_{\lambda}=\overline \pi_{\lambda}\circ\widetilde f_{il}$.

\end{itemize}
}\end{proposition}

\vskip 3mm

Let $(i,l)\in I^{\infty}$ and let $P\in U^-$. Then there exist unique $Q,R\in U^-$ such that
$$[\mathtt a_{il},P]=\tau_{il}(K_i^{l}Q-K_i^{-l}R).$$
Note that $e'_{i,l}(P)=R$ by (\ref{c}). If we set $e''_{i,l}(P)=Q$, then we have
\begin{equation*}
\begin{aligned}
& e'_{i,l}\mathtt b_{jk}=\delta_{ij}\delta_{kl}+q_i^{-kla_{ij}}\mathtt b_{jk}e'_{i,l},\\
& e''_{i,l}\mathtt b_{jk}=\delta_{ij}\delta_{kl}+q_i^{kla_{ij}}\mathtt b_{jk}e''_{i,l},
\end{aligned}
\end{equation*}
and
$$e'_{i,l}e''_{j,k}=q_i^{kla_{ij}}e''_{j,k}e'_{i,l}.$$

%\vskip 3mm

%\begin{remark}{\rm
%Since we have $e''_{i,l}\overline P=\overline{e'_{i,l}P}$ for any $P\in U^-$, hence one can develop the crystal basis theory at $q=0$ by using the operators $e''_{i,l}$.
%}\end{remark}

\vskip 3mm

\begin{definition}
Let $\mathscr B_q(\g)$ be the algebra over $\F(q)$ generated by $e'_{i,l}$, $\mathtt b_{il}$ $(i,l)\in I^{\infty}$ with  defining relations
\begin{equation*}
\begin{aligned}
& e'_{i,l}\mathtt b_{jk}=\delta_{ij}\delta_{kl}+q_i^{-kla_{ij}}\mathtt b_{jk}e'_{i,l},\\
& \sum_{r=0}^{1-la_{ij}}(-1)^r\begin{bmatrix} 1-la_{ij}\\ r
\end{bmatrix}_i
{e'_i}^{1-la_{ij}-r}e'_{j,l}{e'}_i^{r}=0 \ \ \text{for} \ i\in
I^{\text{re}} \ \text {and} \ i \neq (j,l), \\
& \sum_{r=0}^{1-la_{ij}}(-1)^r\begin{bmatrix} 1-la_{ij}\\ r
\end{bmatrix}_i
{\mathtt b_i}^{1-la_{ij}-r}\mathtt b_{j,l}\mathtt b_i^{r}=0 \ \ \text{for} \ i\in
I^{\text{re}} \ \text {and} \ i \neq (j,l), \\
& e'_{i,k}e'_{j,l}-e'_{j,l}e'_{i,k} = \mathtt b_{ik}\mathtt b_{jl}-\mathtt b_{jl}\mathtt b_{ik} =0 \ \ \text{for} \ a_{ij}=0.
\end{aligned}
\end{equation*}
\end{definition}

\vskip 3mm

We call $\mathscr B_q(\g)$ the {\it quantum boson algebra} associated with $\g$. One can show that $\mathscr B_q(\g)$ is a left $U^-_q(\g)$-module by the standard argument in \cite{Kas91}. Furthermore, we have
$$U^-_q(\g)\cong \mathscr B_q(\g)\bigg /\sum_{(i,l)\in I^{\infty}}\mathscr B_q(\g)e'_{i,l}.$$

\vskip 3mm

\begin{lemma}\label{L1}{\rm
For all $P, Q\in U^-$ and $(i,l)\in I^{\infty}$, we have
\begin{equation*}
(P\mathtt b_{il}, Q)_L=\tau_{il}(P,K_i^{l}e''_{i,l}QK_i^{-l})_L.
\end{equation*}
\begin{proof}
By (\ref{c}), we have $K_i^{l}e''_{i,l}Q=\delta_{i,l}(Q)K_i^{l}$ and  hence $K_i^{l}e''_{i,l}QK_i^{-l}=\delta_{i,l}(Q)$. Thus we obtain
$$ (P\mathtt b_{il}, Q)_L=\tau_{il}(P,\delta_{i,l}(Q))_L=\tau_{il}(P,K_i^{l}e''_{i,l}QK_i^{-l})_L$$
as desired.
\end{proof}
}\end{lemma}

\vskip 3mm

Let $*:U_q(\g)\rightarrow U_q(\g)$ be the $\mathbf{F}(q)$-linear  anti-involution given by
$$e^*_{il}=e_{il},
\quad f_{il}^*=f_{il},\quad (q^h)^*=q^{-h}.$$
Note that $**=id$ and $*-=-*$ on $U^\pm$, and $\mathtt a_{il}, \mathtt b_{il}$ are stable under $*$ for any $(i,l)\in I^{\infty}$.

\vskip 3mm

\begin{lemma}\label{L2}{\rm
For any $P, Q\in U^-$, we have $$(P^*,Q^*)_L=(P,Q)_L.$$
\begin{proof}
Note that $e''_{i,l}(Q^*)=K_i^{-l}(e'_{i,l}Q)^*K_i^{l}$ and $e'_{i,l}(Q^*)=K_i^{l}(e''_{i,l}Q)^*K_i^{-l}$. We shall prove this lemma by induction on $|P|$. If $P=1$, our assertion is clear. By Lemma \ref{L1} and the inductive hypothesis, we have
$$\begin{aligned}
& ((P\mathtt b_{il})^*, Q^*)_L=(\mathtt b_{il}P^*,Q^*)_L=\tau_{il}(P^*,e'_{i,l}(Q^*))_L\\
& \phantom{((P\mathtt b_{il})^*, Q^*)_L}=\tau_{il}(P^*,K_i^{l}(e''_{i,l}Q)^*K_i^{-l})_L\\
& \phantom{((P\mathtt b_{il})^*, Q^*)_L}=\tau_{il}(P,K_i^{l}e''_{i,l}QK_i^{-l})_L\\
& \phantom{((P\mathtt b_{il})^*, Q^*)_L}=(P\mathtt b_{il}, Q)_L,
\end{aligned}$$
which proves our claim.
\end{proof}
}\end{lemma}

\vskip 3mm

The following corollary is an immediate consequence of
Lemma \ref{L2} and Proposition \ref{P2}.

\vskip 3mm

\begin{corollary}\label{C1}
$\mathcal L(\infty)^*=\mathcal L(\infty)$.
\end{corollary}

\vskip 3mm

\begin{proposition}{\rm
Let $P, Q\in U^-_q(\g)_{-\beta}$ for $\beta\in Q_+$. If $\lambda\gg 0$,  we have
$$\{Pv_\lambda,Qv_\lambda\}\equiv c(P,Q)_L \ (\text{mod} \ q\mathbb A_0)$$
for some $c\in \mathbb A_0 \backslash q\mathbb A_0$.

\begin{proof}
We use the induction on $|\beta|$. If  $i\in I^{\text{im}}$, we have
$$\begin{aligned}
& \{\mathtt b_{il}Pv_\lambda,Qv_\lambda\}=-\{Pv_{\lambda},K_i^{l}\mathtt a_{il}Qv_\lambda\}\\
& \phantom{\{\mathtt b_{il}Pv_\lambda,Qv_\lambda\}}=-\{Pv_{\lambda}, K_i^{l}(Q\mathtt a_{il}+\tau_{il}(K_i^{l}e''_{i,l}Q-K_i^{-l}e'_{i,l}Q))v_\lambda\}\\
& \phantom{\{\mathtt b_{il}Pv_\lambda,Qv_\lambda\}}=-\tau_{il}\{Pv_{\lambda}, K_i^{2l}e''_{i,l}Qv_\lambda-e'_{i,l}Qv_\lambda\}\\
& \phantom{\{\mathtt b_{il}Pv_\lambda,Qv_\lambda\}}=-\tau_{il}\{Pv_{\lambda}, q_i^{2l(\lambda-\beta)(h_i)}e''_{i,l}Qv_\lambda\}+\tau_{il}\{Pv_\lambda,e'_{i,l}Qv_{\lambda}\},
\end{aligned}$$
where $P\in U^-_{-\beta}$ and $Q\in U^-_{-\beta-l\alpha_i}$. Hence
$$\{\mathtt b_{il}Pv_\lambda,Qv_\lambda\}\equiv \tau_{il}\{Pv_\lambda,e'_{il}Qv_{\lambda}\} \equiv c\tau_{il}(P,e'_{i,l}Q)_L=c(\mathtt b_{il}P, Q)_L \ (\text{mod} \ q\mathbb A_0).$$
 if $i\in I^{\text{re}}$, we have
$$\begin{aligned}
& \{\mathtt b_{i}Pv_\lambda,Qv_\lambda\}=\frac{1}{q_i^2-1}\{Pv_\lambda,K_i\mathtt a_iQv_\lambda\} \\
& \phantom{\{\mathtt b_{i}Pv_\lambda,Qv_\lambda\}}=\frac{1}{q_i^2-1}\{Pv_\lambda,K_i\tau_i(K_ie''_iQ-K_i^{-1}e'_iQ)v_\lambda\}\\
& \phantom{\{\mathtt b_{i}Pv_\lambda,Qv_\lambda\}}=\frac{1}{q_i^2-1}\tau_i\{Pv_\lambda, q_i^{2(\lambda-\beta)(h_i)}e''_{il}Qv_\lambda\}+\frac{1}{q_i^2-1}\tau_i\{Pv_\lambda,e'_iQv_\lambda\},
\end{aligned}$$
where $P\in U^-_{-\beta}$ and $Q\in U^-_{-\beta-\alpha_i}$. Hence
$$\begin{aligned}
& \{\mathtt b_{i}Pv_\lambda,Qv_\lambda\}\equiv \frac{1}{q_i^2-1}\tau_i\{Pv_\lambda,e'_iQv_\lambda\} \equiv \frac{1}{q_i^2-1}c\tau_i(P,e'_{i}Q)_L\\
& \phantom{\{\mathtt b_{i}Pv_\lambda,Qv_\lambda\}}=\frac{1}{q_i^2-1}c(\mathtt b_{i}P, Q)_L \ (\text{mod} \ q\mathbb A_0),
\end{aligned}$$
which completes the proof.
\end{proof}
}\end{proposition}

\vskip 3mm

\begin{corollary}{\rm
If $\lambda\gg 0$ and $Pv_\lambda\in \mathcal L(\lambda)$, then $P^*v_\lambda\in \mathcal L(\lambda)$.
\begin{proof}
If $Pv_\lambda\in\mathcal L(\lambda)$, then $\{Pv_\lambda,Pv_\lambda\}\in \mathbb A_0$ by Proposition \ref{P3}. Since $\{Pv_\lambda,Pv_\lambda\}\equiv c(P,P)_L \ (\text{mod} \ q\mathbb A_0)$  for some $c \in \mathbb{A}_{0} \setminus q \, \mathbb{A}_{0}$, we have $(P,P)_L\in \mathbb A_0$. Hence $P\in \mathcal L(\infty)$ and $P^*\in \mathcal L(\infty)$ by Proposition \ref{P2} and Corollary \ref{C1}. Now Proposition \ref{P4} yields $\pi_{\lambda}(\mathcal L(\infty))=\mathcal L(\lambda)$. 
Thus we get $P^*v_{\lambda}\in \mathcal L(\lambda)$ by applying $\pi_\lambda$.
\end{proof}
}\end{corollary}

\vskip 8mm

\section{$\mathbb A$-form of $U^-_q(\g)$}

\vskip 2mm

Let $\mathbb A=\F[q,q^{-1}]$ and $\mathbb A_\infty=\{f\in \F(q) \mid f \ \text{is regular at} \ q=\infty\}$. We denote by $U^-_{\mathbb A}(\g)$ the $\mathbb A$-subalgebra of $U_q(\g)$ generated by $\mathtt b_i^{(n)}$ $(i\in I^{\text{re}}, n\geq 0)$ and $\mathtt b_{il}$ $(i\in I^{\text{im}}, l\geq1)$.

\vskip 3mm

For each $i\in I^{\text{re}}$, set
\begin{equation}\label{eq:new}
A_i=\mathtt a_i/\tau_i(q_i-q_i^{-1}),
\end{equation}
which yields the following commutation relation
\begin{equation} \label{eq:comm2}
A_i\mathtt b_i-\mathtt b_iA_i=\frac{K_i-K_i^{-1}}{q_i-q_i^{-1}}.
\end{equation}
For $i\in I^{\text{im}}$ and $l\geq1$, set $A_{il}=\mathtt a_{il}/\tau_{il}$.  Then we have
\begin{equation}\label{eq:comm3}
A_{il}\mathtt b_{il}-\mathtt b_{il}A_{il}=K_i^{l}-K_i^{-l}.
\end{equation}

\vskip 3mm

 Let $U_{\mathbb A}(\g)$ be the $\mathbb A$-subalgebra of $U_q(\g)$ generated by $A_i^{(n)},\mathtt b_i^{(n)}$ $(i\in I^{\text{re}}, n\geq 0)$, $A_{il}, \mathtt b_{il}$ $(i\in I^{\text{im}}, l\geq1)$ and $q^h$ $(h \in P^{\vee})$,  $\left\{ \begin{matrix}K_iq_i^n \\ m \end{matrix}\right\}_i$ $(i \in I^{\text{re}}, m\in \Z_{\geq 0}, n\in \Z)$,
where
\begin{equation} \label{df3}
 \left\{ \begin{matrix}K_iq_i^n \\ m \end{matrix}\right\}_i=\frac{1}{[m]_i!}\prod_{s=1}^m\frac{K_iq_i^{n+1-s}-K_i^{-1}q_i^{-n-1+s}}{q_i-q_i^{-1}}.
\end{equation}

Let $U_{\mathbb A}^+(\g)$ (resp. $U_{\mathbb A}^0(\g)$) be the $\mathbb A$-subalgebra of $U_q(\g)$ generated by $A_i^{(n)}$ $(i\in I^{\text{re}})$ and $A_{il}$ $(i\in I^{\text{im}},  l \ge 1)$
(resp. $q^h, \left\{ \begin{matrix}K_iq_i^n \\ m \end{matrix}\right\}_i$ for $h\in P^{\vee}, m\in \Z_{\geq 0}, n\in \Z \ \text{and} \ i\in I^{\text{re}}$ ). Then using the commutations relations \eqref{eq:comm2}, \eqref{eq:comm3} and the definition \eqref{df3}, one can prove that $U_{\mathbb A}(\g)$ has
the triangular decomposition (see also \cite[Section 1]{Kas91}, \cite[Exercise 3.6]{HK2002})
$$U_{\mathbb A}(\g)\cong U^-_{\mathbb A}(\g)\otimes U^0_{\mathbb A}(\g)\otimes U^+_{\mathbb A}(\g).$$

\vskip 3mm

Let $\lambda\in P^{+}$ and consider an $\F$-linear automorphism $^-:V(\lambda)\rightarrow V(\lambda)$ given by $Pv_\lambda\mapsto\overline Pv_\lambda$ for $P\in U_q(\g)$. Set $\mathcal L(\lambda)^-=\overline{\mathcal L(\lambda)}$. Then $\mathcal L(\lambda)$ (resp. $\mathcal L(\lambda)^-$) is a free $\mathbb A_0$-lattice (resp. free $\mathbb A_\infty$-lattice) of $V(\lambda)$.

\vskip 3mm

Since $$\left\{ \begin{matrix}K_iq_i^n \\ m \end{matrix}\right\}_iv_\lambda=\begin{bmatrix}\lambda(h_i)+n\\ m\end{bmatrix}_iv_\lambda\in \Z[q,q^{-1}]v_\lambda,$$
we get $U^0_\mathbb A(\g)v_\lambda=\mathbb Av_\lambda$. This
leads us to give  the following definition
$$V(\lambda)^{\mathbb A}:=U_{\mathbb A}(\g)v_\lambda=U^-_{\mathbb A}(\g)v_\lambda.$$
Note that $\overline {\mathtt b}_{il}=\mathtt b_{il}$ for all $(i,l)\in I^{\infty}$. Hence $U^-_{\mathbb A}(\g)$ and $V(\lambda)^{\mathbb A}$ are stable  under $-$. Also, since $U^-_{\mathbb A}(\g)$ is graded by $Q_-$, we have  $V(\lambda)^{\mathbb A}=\bigoplus_{\mu\leq \lambda}V(\lambda)_{\mu}^{\mathbb A}$, where $V(\lambda)_{\mu}^{\mathbb A}=V(\lambda)^{\mathbb A}\cap V(\lambda)_\mu$.

\vskip 3mm

\vskip 3mm

Fix $i \in I$. In \cite{Bozec2014c}, Bozec proved that
every $u \in U_{q}^{-}(\g)$ has the following decomposition.

\begin{equation} \label{eq:decomp}
u = \begin{cases}
& \sum_{n\geq 0}\mathtt b_i^{(n)}u_n \quad \text{with}\ i\in I^{\text{re}} \ \text{and}\ e'_iu_n=0 \ \text{for all} \ n\geq 0,\\
& \sum_{\mathbf c \in \mathscr C_i}\mathtt b_{i,\mathbf c}u_{\mathbf c} \quad \text{with}\ i\in I^{\text{im}} \ \text{and}\ e'_{il}u_{\mathbf c}=0 \ \text{for all} \ l>0, \mathbf c\in \mathscr C_i.
\end{cases}
\end{equation}

\vskip 3mm

\begin{lemma}

{\rm
For each $i \in I $ and $u \in U_{q}^{-}(\g)$, consider the decomposition \eqref{eq:decomp}.
If $u\in U^-_{\mathbb A}(\g)$, then all  $u_n, u_{\mathbf c} \in U^-_{\mathbb A}(\g)$.

\begin{proof}
We first prove that $e'_{i,l}U^-_{\mathbb A}(\g)\subseteq U^-_{\mathbb A}(\g)$ for all $(i,l)\in I^{\infty}$.
\vskip 2mm
Since
$e'_{i,l}\mathtt b_{jk}=\delta_{ij}\delta_{kl}+q_i^{-kla_{ij}}\mathtt b_{jk}e'_{i,l}, $
we have
$$e'_i\mathtt b_i=1+q_i^{-2}\mathtt b_ie'_i \ \ \text{for} \ i\in I^{\text{re}}.$$
It follows that
$$e'_i\mathtt b_i^{(n)}=q_i^{1-n}\mathtt b_i^{(n-1)}+q_i^{-2n}\mathtt b_i^{(n)}e'_i.$$

\vskip 2mm
Furthermore, by a direct calculation,  we have
$$e'^{n}_i\mathtt b_i^{(m)}=\sum_{k=0}^{n}q_i^{-2nm+(m+n)k-k(k-1)/2}\begin{bmatrix}n\\k\end{bmatrix}_i\mathtt b_i^{(m-k)}e'^{n-k}_i,$$
where $\mathtt b_i^{(r)}=0$ if $r<0$. These imply our assertion.

\vskip 3mm

For $i\in I^{\text{re}}$, let $$P=\sum_{n\geq 0}(-1)^nq_i^{-n(n-1)/2}\mathtt b_i^{(n)}e'^{n}_i.$$
Then we obtain  (cf. \cite[Section 3.2]{Kas91}):
\begin{itemize}
\item [(i)] $P\mathtt b_i=0$, \, $e'_iP=0$,
\item [(ii)] $\sum_{n\geq 0}q_i^{n(n-1)/2} \, \mathtt b_i^{(n)}Pe'^{n}_i=1$,
\item [(iii)] $Pe'^{n}_iu=q_i^{-n(n-1)/2}u_n$ for $u\in U^-_q(\g)$.
\end{itemize}

\vskip 2mm

Hence, if $u\in U^-_{\mathbb A}(\g)$, then $u_n\in U^-_{\mathbb A}(\g)$ for all $n\geq 0$.

\vskip 3mm

For $i\in I^{\text{im}}$, we use a similar argument in \cite[Proposition 3.11]{Bozec2014b}. Assume that $u\in U^-_{\mathbb A}(\g)$  has the form $u=m\mathtt b_{i,\mathbf c}m'$ for some $\mathbf c\in \mathscr C_i$ and homogeneous  elements  $m, m'\in \mathcal K_i\cap U^-_{\mathbb A}(\g)$, where $\mathcal K_i=\bigcap_{l>0}\text{ker}e'_{i,l}$. We shall show that $u$ can be written into the form
$$u=\sum_{\mathbf c'\in \mathscr C_i}\mathtt b_{i,\mathbf c'}u_{\mathbf c'} \ \ \text{with}
 \ u_{\mathbf c'}\in \mathcal K_i\cap U^-_{\mathbb A}(\g).$$

If $|\mathbf c|=0$, then $u=mm'\in \mathcal K_i\cap U^-_{\mathbb A}(\g)$. Otherwise, we have
$$u=(m\mathtt b_{ic_1}-q^{c_1(|m|,\alpha_i)}\mathtt b_{ic_1}m)\mathtt b_{i,\mathbf c\backslash c_1}m'+q^{c_1(|m|,\alpha_i)}\mathtt b_{ic_1}m\mathtt b_{i,\mathbf c\backslash c_1}m',$$
where $m\mathtt b_{ic_1}-q^{c_1(|m|,\alpha_i)}\mathtt b_{ic_1}m\in \mathcal K_i\cap U^-_{\mathbb A}(\g)$. Now our claim follows by using the induction on $|\mathbf c|$.

\vskip 3mm

We next show that if $u\in U^-_{\mathbb A}(\g)$, then $u$ can be written into the form
$$u=\sum_{\mathbf c \in \mathscr C_i}\mathtt b_{i,\mathbf c}u_{\mathbf c} \  \text{ with}
\ u_{\mathbf c}\in \mathcal K_i\cap U^-_{\mathbb A}(\g).$$ We will  use the induction on $-|u|$.

\vskip 3mm

Assume that $u$ is a monomial in $U^-_{\mathbb A}(\g)$. Then  there exists some monomial $u'\in U^-_{\mathbb A}(\g)$ such that $u=\mathtt b_j^{(n)}u'$ for some $j\in I^{\text{re}}$ or $u=\mathtt b_{jl}u'$ for some $j\in I^{\text{im}}$. By induction hypothesis, $u'=\sum_{\mathbf c \in \mathscr C_i}\mathtt b_{i,\mathbf c}u_{\mathbf c}$ with $u_{\mathbf c}\in \mathcal K_i\cap U^-_{\mathbb A}(\g)$. If $j\neq i$,
then $u=\sum_{\mathbf c \in \mathscr C_i}\mathtt b_j^{(n)}\mathtt b_{i,\mathbf c} \, u_{\mathbf c}$ or $u=\sum_{\mathbf c \in \mathscr C_i}\mathtt b_{jl}\mathtt b_{i,\mathbf c}u_{\mathbf c}$ is of the form $m\mathtt b_{i,\mathbf c}m'$ with $m, m'\in \mathcal K_i\cap U^-_{\mathbb A}(\g)$. If $i=j$, then $u=\sum_{\mathbf c \in \mathscr C_i}\mathtt b_{i,(l,\mathbf c)}u_{\mathbf c}$ is already in the form  we wanted.

\vskip 3mm

Thus, our assertion  follows from the uniquesness of the decomposition.
\end{proof}
}\end{lemma}

\vskip 3mm

Define
\begin{equation*}
\begin{aligned}
& (\mathtt b_i^nU_q^-(\g))^\mathbb A:=\mathtt b_i^nU_q^-(\g)\cap U^-_{\mathbb A}(\g) \quad \text{for} \ i\in I^{\text{re}} \ \text{and}\ n\geq 1,\\
& (\mathtt b_{i,\mathbf c}U_q^-(\g))^\mathbb A:=\mathtt b_{i,\mathbf c}U_q^-(\g)\cap U^-_{\mathbb A}(\g) \quad \text{for} \ i\in I^{\text{im}} \ \text{and}\ \mathbf c \in \mathscr C_i\backslash \{0\}.
\end{aligned}
\end{equation*}

By  the above lemma, $U^-_{\mathbb A}(\g)$ is stable under the Kashiwara operators $\widetilde{e}_{il}$
and $\widetilde{f}_{il}$. Moreover, we can prove the following corollary easily.

\vskip 3mm

\begin{corollary}\
{\rm
\begin{itemize}
\item [(i)] For $i\in I^{\text{re}}$, $(\mathtt b_i^nU_q^-(\g))^\mathbb A=\sum_{k\geq n}\mathtt b_i^{(k)}U^-_{\mathbb A}(\g)=\bigoplus_{k\geq n}  \mathtt b_i^{(k)} (U^-_{\mathbb A}(\g)\cap \text{ker} e'_i)$.\\
    For $i\in I^{\text{im}}$, $(\mathtt b_{i,\mathbf c}U_q^-(\g))^\mathbb A= b_{i,\mathbf c}U^-_{\mathbb A}(\g)=\bigoplus_{\mathbf c'\in\mathscr C_i}\mathtt b_{i,(\mathbf c,\mathbf c')}(U^-_{\mathbb A}(\g)\cap \mathcal K_i)$.
\item [(ii)] For $i\in I^{\text{re}}$, $(\mathtt b_i^nV(\lambda))^{\mathbb A}:=(\mathtt b_i^nU_q^-(\g))^\mathbb Av_\lambda=\sum_{k\geq n}\mathtt b_i^{(k)}V(\lambda)^{\mathbb A}$. \\
    For $i\in I^{\text{im}}$, $(\mathtt b_{i,\mathbf c}V(\lambda))^\mathbb A:=(\mathtt b_{i,\mathbf c}U_q^-(\g))^\mathbb Av_\lambda=\mathtt b_{i,\mathbf c}V(\lambda)^{\mathbb A}$.
\end{itemize}
}\end{corollary}

\vskip 3mm

\begin{proposition}\cite[Proposition 6.1.3]{Kas91}, \cite[Lemma 6.3.7]{HK2002} {\rm
Let $i\in I^{\text{re}}$ and $\mu\in P$.
\begin{itemize}
\item [(i)] For $u \in V(\lambda)_\mu$ with $n=-\mu(h_i)\geq 1$, we have
$$u=\sum_{k\geq n}(-1)^{k-n}\begin{bmatrix}k-1\\k-n\end{bmatrix}_i\mathtt b_i^{(k)}A_i^{(k)}u.$$
\item [(ii)] If $n=-\mu(h_i)\geq 0$, then we have
$$V(\lambda)_{\mu}^{\mathbb A}=\sum_{k\geq n}\mathtt b_i^{(k)}V(\lambda)^{\mathbb A}_{\mu+k\alpha_i}=(\mathtt  b_i^{ n} V(\lambda))_\mu^{\mathbb A}.$$
\end{itemize}
}\end{proposition}

\vskip 8mm

\section{Existence and uniqueness of global bases}

\vskip 2mm

Let $V$ be a finite-dimensional vector space over $\F(q)$, $M$ be  an $\mathbb A$-submodule of $V$,  and $L_0$ (resp. $L_\infty$) be a free $\mathbb A_0$-submodule (resp. free $\mathbb A_\infty$-submodule) of $V$ such that $V\cong \F(q)\otimes_{\mathbb A_0}L_0 \cong \F(q)\otimes_{\mathbb A_\infty}L_\infty$.

\vskip 3mm

\begin{lemma}\cite[Lemma 7.1.1]{Kas91}\label{K1} {\rm
 Let $V, M, L_0, L_\infty$  be as above.
\begin{itemize}
\item [(i)] Assume that the canonical map $M\cap L_0 \cap L_\infty\rightarrow M\cap L_0/M\cap qL_0$ is an isomorphism.  Then
    $$M\cap L_0\cong \F[q]\otimes_{\F}(M\cap L_0 \cap L_\infty),$$
    $$M\cap L_\infty \cong \F[q^{-1}]\otimes_{\F}(M\cap L_0 \cap L_\infty),$$
    $$M\cong \mathbb A\otimes_{\F}(M\cap L_0 \cap L_\infty),$$
    $$M\cap L_0 \cap L_\infty\cong (M\cap L_\infty)/(M\cap q^{-1}L_\infty),$$
    $$M\cap L_0 \cap L_\infty\simeq (\F(q)\otimes_{\mathbb A}M)\cap L_0/(\F(q)\otimes_{\mathbb A}M)\cap qL_0.$$
\item [(ii)] Let $E$ be  an  $\F$-vector space and $\varphi:E\rightarrow M\cap L_0 \cap L_\infty$ a homomorphism. Assume that
$M=\mathbb A\varphi(E)$ and $E\rightarrow L_0/qL_0, E\rightarrow L_\infty/q^{-1}L_\infty$ are injective. Then $E\rightarrow M\cap L_0 \cap L_\infty \rightarrow M\cap L_0/M \cap qL_0$ is an isomorphism.
\end{itemize}
}\end{lemma}

\vskip 3mm

\begin{lemma}\cite[Lemma 7.1.2]{Kas91}\label{K2} {\rm
 Let $V, M, L_0, L_\infty$ be as above  and  let $N$ be an $\mathbb A$-submodule of $M$.  Assume that
\begin{itemize}
\item [(1)] $N\cap L_0\cap L_\infty\cong N\cap L_0/N\cap qL_0$.
\item [(2)] There exist an $\F$-vector space $F$ and a homomorphism $\varphi:F\rightarrow M\cap(L_0+N)\cap (L_\infty+N)$ such that\\
    (a) $M=\mathbb A\varphi(F)+N$,\\
    (b) the induced homomorphisms  $\varphi_0:F\rightarrow (L_0+N) /  (qL_0+N)$ and $\varphi_\infty:F\rightarrow (L_\infty+N) / (q^{-1}L_\infty+N)$ are injective.
\end{itemize}

\vskip 2mm

Then the following statements hold.

\vskip 2mm

\begin{itemize}
\item [(i)] $M\cap L_0\cap L_\infty\rightarrow M\cap L_0/M\cap qL_0$ is an isomorphism.
\item [ii)] $M\cap L_0/M\cap qL_0 \cong F\oplus (N\cap L_0/N\cap qL_0)$.
\end{itemize}
}\end{lemma}

\vskip 3mm

For $r\geq 0$, set $$Q_+(r)=\{\alpha\in Q_+\mid |\alpha|\leq r\}.$$

 We will prove the following  inductive statements on $r\geq 0$.
\vskip 3mm
\begin{itemize}
\item [A($r$)]: For any $\alpha\in Q_+(r)$, we have the following canonical isomorphism
$$U^-_{\mathbb A}(\g)_{-\alpha}\cap \mathcal L(\infty)\cap\mathcal L(\infty)^- \xrightarrow{\sim} \frac{U^-_{\mathbb A}(\g)_{-\alpha}\cap \mathcal L(\infty)}{U^-_{\mathbb A}(\g)_{-\alpha}\cap q\mathcal L(\infty)}\xrightarrow{\sim} \mathcal L(\infty)_{-\alpha}/q\mathcal L(\infty)_{-\alpha}.$$
\end{itemize}
We denote by $G_\infty$ the inverse of this isomorphism.

\vskip 2mm

\begin{itemize}
\item [B($r$)]: For any $\alpha\in Q_+(r)$ and $\lambda\in P^{+}$, we have the following canonical isomorphism
$$V(\lambda)^{\mathbb A}_{\lambda-\alpha}\cap \mathcal L(\lambda)\cap\mathcal L(\lambda)^- \xrightarrow{\sim} \frac{V(\lambda)^{\mathbb A}_{\lambda-\alpha}\cap \mathcal L(\lambda)}{V(\lambda)^{\mathbb A}_{\lambda-\alpha}\cap q\mathcal L(\lambda)}\xrightarrow{\sim} \mathcal L(\lambda)_{\lambda-\alpha}/q\mathcal L(\lambda)_{\lambda-\alpha}.$$
\end{itemize}
We denote by $G_\lambda$ the inverse of this isomorphism.

\vskip 2mm

\begin{itemize}
\item [C($r$)]: For $\alpha\in Q_+(r)$, $(i,l)\in I^{\infty}$, and $n\geq 0$, assume that $b\in \widetilde{f}_{il}^n \, \mathcal B(\infty)_{-\alpha+l n\, \alpha_i}$.  Then we have
    $$G_\infty(b)\in \mathtt b_{il}^n \, U_q^-(\g).$$
\end{itemize}

\vskip 3mm

If $r=0$, our assertions are obvious. Now we assume that A($r-1$), B($r-1$) and C($r-1$) are true.  Then  Lemma \ref{K1} and Proposition \ref{P4} imply the following result.

\vskip 3mm

\begin{lemma}{\rm
For $\alpha\in Q_+(r-1)$, we have
$$\begin{aligned}
& U^-_{\mathbb A}(\g)_{-\alpha}\cap \mathcal L(\infty)=\bigoplus_{b\in \mathcal B(\infty)_{-\alpha}}\F[q]G_\infty(b), \quad
U^-_{\mathbb A}(\g)_{-\alpha}=\bigoplus_{b\in \mathcal B(\infty)_{-\alpha}}\mathbb AG_\infty(b),\\
& V(\lambda)^{\mathbb A}_{\lambda-\alpha}\cap \mathcal L(\lambda)=\bigoplus_{b\in \mathcal B(\lambda)_{\lambda-\alpha}}\F[q]G_\lambda(b), \quad
V(\lambda)^{\mathbb A}_{\lambda-\alpha}=\bigoplus_{b\in \mathcal B(\lambda)_{\lambda-\alpha}}\mathbb AG_\lambda(b),
\end{aligned}$$
and
$$G_\infty(b)v_\lambda=G_\lambda(\overline \pi_\lambda (b)).$$
}\end{lemma}

\vskip 3mm

\begin{lemma}{\rm
For $\alpha\in Q_+(r-1)$, $b\in \mathcal B(\infty)_{-\alpha}$ (resp. $b\in \mathcal B(\lambda)_{\lambda-\alpha}$), we have $\overline{G_\infty(b)}=G_\infty(b)$ (resp. $\overline{G_\lambda(b)}=G_\lambda(b)$).
\begin{proof}
Let $Q=(G_\infty(b)-\overline{G_\infty(b)})/(q-q^{-1})$. Then we have $Q\in U^-_{\mathbb A}(\g)_{-\alpha}\cap q\mathcal L(\infty)\cap \mathcal L(\infty)^-$ since $1/(q-q^{-1})\in q\mathbb A_0$.
\end{proof}
}\end{lemma}

\vskip 3mm

Let $i\in I^{\text{iso}}$, $\lambda \in P^{+}$ and $\alpha \in Q_{+}$.
For each partition  $\mathbf {c}=(1^{l_1}2^{l_2}\cdots k^{l_k}\cdots)$,
we define
\begin{equation*}
\begin{aligned}
& ({\mathtt b}_{i, \mathbf{c}} * U_{q}^{-}(\g))^{\mathbb A}_{-\alpha}
:= \sum_{k \ge 1} ({\mathtt b}_{i,k}^{l_k} U_{q}^{-}(\g))^{\mathbb A}_{-\alpha}
= \sum_{k \ge 1} {\mathtt b}_{i,k}^{l_k} (U_{\mathbb A}^{-}(\g)_{-\alpha + k l_{k} \alpha_{i}}), \\
& ({\mathtt b}_{i, \mathbf{c}} * V(\lambda))^{\mathbb A}_{\lambda - \alpha} :
= \sum_{k \ge 1} ({\mathtt b}_{i, k}^{l_k} V(\lambda))^{\mathbb A}_{\lambda - \alpha}
=\sum_{k \ge 1} {\mathtt b}_{i, k}^{l_k}(V(\lambda)^{\mathbb A}_{\lambda - \alpha + k l_k \alpha_{i}}).
\end{aligned}
\end{equation*}

Here $(\mathtt b_{i,k}^{l_k}U^-_q(\g))^{\mathbb A}_{-\alpha}
= ({\mathtt b}_{i,k}^{l_k} V(\lambda))^{\mathbb A}_{\lambda - \alpha} = 0   $ if $l_k=0$.

\vskip 3mm

\begin{proposition}\label{M}{\rm
Let $\alpha\in Q_+(r)$ and $\lambda\in P^+$.
\begin{itemize}
\item [(i)] For $i\in I^{\text{re}}$ and $n\geq 1$, we have
$$(\mathtt b_i^nV(\lambda))^{\mathbb A}_{\lambda-\alpha}\cap \mathcal L(\lambda)\cap \mathcal L(\lambda)^-\xrightarrow{\sim}\frac{(\mathtt b_i^nV(\lambda))^{\mathbb A}_{\lambda-\alpha}\cap \mathcal L(\lambda)}{(\mathtt b_i^nV(\lambda))^{\mathbb A}_{\lambda-\alpha}\cap q\mathcal L(\lambda)}\cong \bigoplus_{b\in \mathcal B(\lambda)_{\lambda-\alpha}\cap \widetilde f_i^n\mathcal B(\lambda)}\F b.$$
\item [(ii)] For $i\in I^{\text{im}}\backslash I^{\text{iso}}$ and any composition $\mathbf c$ with $|\mathbf c|\neq 0$, we have
$$(\mathtt b_{i,\mathbf c}V(\lambda))^{\mathbb A}_{\lambda-\alpha}\cap \mathcal L(\lambda)\cap \mathcal L(\lambda)^-\xrightarrow{\sim}\frac{(\mathtt b_{i,\mathbf c}V(\lambda))^{\mathbb A}_{\lambda-\alpha}\cap \mathcal L(\lambda)}{(\mathtt b_{i,\mathbf c}V(\lambda))^{\mathbb A}_{\lambda-\alpha}\cap q\mathcal L(\lambda)}\cong \bigoplus_{b\in \mathcal B(\lambda)_{\lambda-\alpha}\cap \widetilde f_{i,\mathbf c}\mathcal B(\lambda)}\F b.$$
\item [(iii)] For $i\in I^{\text{iso}}$ and any partition $\mathbf c=1^{l_1}2^{l_2}\cdots k^{l_k}\cdots$, we have
$$(\mathtt b_{i,\mathbf c}*V(\lambda))^{\mathbb A}_{\lambda-\alpha}\cap \mathcal L(\lambda)\cap \mathcal L(\lambda)^-\xrightarrow{\sim}\frac{(\mathtt b_{i,\mathbf c}*V(\lambda))^{\mathbb A}_{\lambda-\alpha}\cap \mathcal L(\lambda)}{(\mathtt b_{i,\mathbf c}*V(\lambda))^{\mathbb A}_{\lambda-\alpha}\cap q\mathcal L(\lambda)}\cong \bigoplus_{b\in B(\lambda)_{\lambda-\alpha}\cap(\widetilde f_{i,\mathbf c}*\mathcal B(\lambda))}\F b,$$
where $\widetilde f_{i,\mathbf c}*\mathcal B(\lambda):=\bigcup_{k\geq 1}  {\widetilde f}_{i,k}^{l_k}\mathcal B(\lambda)$.
\end{itemize}
\begin{proof}
Our assertion (i) has been proved in \cite[Proposition 7.4.1]{Kas91}.

\vskip 3mm

Assume that
$i\in I^{\text{im}}\backslash I^{\text{iso}}$ and let $\mathbf c\in {\mathscr C}_i$ with $|\mathbf c|=n>0$.
Recall that $(\mathtt b_{i,\mathbf c}V(\lambda))^{\mathbb A}_{\lambda-\alpha}=\mathtt b_{i,\mathbf c}(V(\lambda)^{\mathbb A}_{\lambda-\alpha+n\alpha_i})$.
If $\langle h_{i}, \lambda-\alpha+n\alpha_i \rangle=0$, then $(\mathtt b_{i,\mathbf c}V(\lambda))^{\mathbb A}_{\lambda-\alpha}=0$, and hence our assertion is trivial.
Thus we may assume that $\langle h_{i}, \lambda-\alpha+n\alpha_i \rangle >0$.
 In this case,
 for any $b\in \mathcal B(\lambda)_{\lambda-\alpha+n\alpha_i}$, we have $\widetilde f_{i,\mathbf c}b\neq 0$.

By B($r-1$), we have $V(\lambda)^{\mathbb A}_{\lambda-\alpha+n\alpha_i}=\bigoplus_{b\in\mathcal B(\lambda)_{\lambda-\alpha+n\alpha_i}}\mathbb AG_\lambda(b)$. Hence
$$(\mathtt b_{i,\mathbf c}V(\lambda))^{\mathbb A}_{\lambda-\alpha}=\sum_{b\in \mathcal B(\lambda)_{\lambda-\alpha+n\alpha_i}}\mathbb A\mathtt b_{i,\mathbf c}G_\lambda(b).$$
Let $F=\sum_{b\in \mathcal B(\lambda)_{\lambda-\alpha+n\alpha_i}}\F \mathtt b_{i,\mathbf c}G_\lambda(b)$. We first show that $F$ is a direct sum. Assume that
$$\sum_{b\in \mathcal B(\lambda)_{\lambda-\alpha+n\alpha_i}}\beta_b \mathtt b_{i,\mathbf c}G_\lambda(b)=0\quad \text{for some} \ \beta_b\in \F.$$
Since $\widetilde f_{i,\mathbf c}G_\lambda(b)=\mathtt b_{i,\mathbf c}G_\lambda(b)$ and $G_\lambda(b)\equiv b \ (\text{mod}\ q\mathcal L(\lambda))$ for any $b\in \mathcal B(\lambda)_{\lambda-\alpha+n\alpha_i}$, we obtain
$$\sum_{b\in \mathcal B(\lambda)_{\lambda-\alpha+n\alpha_i}}\beta_b \widetilde f_{i,\mathbf c}b=0{\color{red}.}$$
By applying $\widetilde e_{i,\widetilde{\mathbf c}}$,  we get $\sum_{b\in \mathcal B(\lambda)_{\lambda-\alpha+n\alpha_i}}\beta_b b=0$,  which implies $\beta_b=0$ for all $b\in \mathcal B(\lambda)_{\lambda-\alpha+n\alpha_i}$.

\vskip 3mm

Let $N=0$, $M=(\mathtt b_{i,\mathbf c}V(\lambda))^{\mathbb A}_{\lambda-\alpha}$, $L_0=\mathcal L(\lambda)_{\lambda-\alpha}$ and $L_\infty=\mathcal L(\lambda)_{\lambda-\alpha}^-$. Set $\varphi:F\rightarrow M\cap L_0\cap L_\infty$ be the $\F$-linear map given by
 $$\mathtt b_{i,\mathbf c}G_\lambda(b) \longmapsto \mathtt b_{i,\mathbf c}G_\lambda(b)=\widetilde f_{i,\mathbf c} G_\lambda(b).$$

 Then, it is easy to check $F,N,M,L_0,L_\infty$ and $\varphi$ satisfy the conditions in Lemma \ref{K2}, and hence we get
$$M\cap L_0\cap L_\infty\xrightarrow{\sim}M\cap L_0/M\cap qL_0\cong \bigoplus_{b\in \mathcal B(\lambda)_{\lambda-\alpha+n\alpha_i}}\F\widetilde f_{i,\mathbf c}b=\bigoplus_{b\in \mathcal B(\lambda)_{\lambda-\alpha}\cap \widetilde f_{i,\mathbf c}\mathcal B(\lambda)}\F b$$
as desired.

\vskip 2mm

Now, we shall prove (iii). Let $i\in I^{\text{iso}}$.  If $l_k$  is sufficient large, then $(\mathtt b_{ik}^{l_k}U^-_q(\g))^{\mathbb A}_{-\alpha}=0$. Hence we can use descending induction on $N = \sum_{k\geq 1}l_k$.
 Without loss of generality, we may assume that $l_1\neq 0$ and $\langle h_{i}, \lambda-\alpha \rangle >0$.
Then, by A($r-1$) and B($r-1$), we  have
$$(\mathtt b_{i,\mathbf c}*V(\lambda))^{\mathbb A}_{\lambda-\alpha}= \mathtt b_{i,1}^{l_1}V(\lambda)^{\mathbb A}_{\lambda-\alpha+l_1\alpha_i}+\sum_{k\geq 2}(\mathtt b_{i,k}^{l_k}V(\lambda))^{\mathbb A}_{\lambda-\alpha}$$
and
$$V(\lambda)^{\mathbb A}_{\lambda-\alpha+l_1\alpha_i}=\bigoplus_{b\in \mathcal B(\lambda)_{\lambda-\alpha+l_1\alpha_i}}\mathbb AG_\lambda(b)=\bigoplus_{\substack{b\in \mathcal B(\infty)_{-\alpha+l_1\alpha_i}\\ \overline \pi_\lambda (b)\neq 0}}\mathbb AG_\infty(b)v_\lambda.$$

Let $b\in \mathcal B(\infty)_{-\alpha+l_1\alpha_i}$ with $\widetilde e_{i1}b\neq 0$. Then $b\in \widetilde f_{i1}\mathcal B(\infty)\cap\mathcal B(\infty)_{-\alpha+l_1\alpha_i}$, which implies $G_\infty(b)\in \mathtt b_{i1}U^-_q(\g)\cap U^-_{\mathbb A}(\g)$ by C($r-1$). Hence $\mathtt b_{i1}^{l_1}G_\infty(b)v_\lambda\in (\mathtt b_{i1}^{l_1+1}V(\lambda))^{\mathbb A}_{\lambda-\alpha}$. For $k\neq 1$ with $l_k\neq 0$, if $b\in \widetilde f_{ik}^{l_k}\mathcal B(\infty)$, then $G_\infty(b)\in \mathtt b_{ik}^{l_k}U^-_q(\g)\cap U^-_{\mathbb A}(\g)$ by C($r-1$). Therefore $\mathtt b_{i,1}^{l_1}G_\infty(b)v_\lambda\in(\mathtt b_{i,k}^{l_k}V(\lambda))^{\mathbb A}_{\lambda-\alpha}$.  Hence we have
$$(\mathtt b_{i,\mathbf c}*V(\lambda))^{\mathbb A}_{\lambda-\alpha}=\sum_{b\in S}\mathbb A\mathtt b_{i,1}^{l_1}G_\infty(b)v_\lambda+(\mathtt b_{i,\mathbf c\cup \{1\}}*V(\lambda))^{\mathbb A}_{\lambda-\alpha},$$
where
\begin{equation}\label{S}
\begin{aligned}
& S=\{b\in\mathcal B(\infty)_{-\alpha+l_1\alpha_i}\mid \overline \pi_\lambda (b)\neq 0,\ \widetilde e_{i,1}b=0,\  b\notin \bigcup_{k\geq 2}\widetilde f_{i,k}^{l_k}\mathcal B(\infty)\} \\
& \phantom{S}\xrightarrow{\substack{\overline \pi_\lambda\\ \sim}} \{b\in\mathcal B(\lambda)_{\lambda-\alpha+l_1\alpha_i}\mid \widetilde e_{i,1}b=0,\  b\notin \bigcup_{k\geq 2}\widetilde f_{i,k}^{l_k}\mathcal B(\lambda)\}\\
& \phantom{S}\xrightarrow{\substack{\widetilde f_{i,1}^{l_1}\\ \sim}} \mathcal B(\lambda)_{\lambda-\alpha}\cap \left(\widetilde f_{i,1}^{l_1}\mathcal B(\lambda)\backslash (\widetilde f_{i,\mathbf c\cup \{1\}}*\mathcal B(\lambda))\right).
\end{aligned}
\end{equation}
The last isomorphism follows from the fact that $\widetilde f_{il}\widetilde e_{il'}=\widetilde e_{il'}\widetilde f_{il}$ and $\widetilde f_{il}\widetilde f_{il'}=\widetilde f_{il'}\widetilde f_{il}$ for any $l,l'\geq 1$ with $l\neq l'$.

Let $V=V(\lambda)_{\lambda-\alpha}$, $M=(\mathtt b_{i,\mathbf c}*V(\lambda))^{\mathbb A}_{\lambda-\alpha}$, $N=(\mathtt b_{i,\mathbf c\cup\{1\}}*V(\lambda))^{\mathbb A}_{\lambda-\alpha}$, $L_0=\mathcal L(\lambda)_{\lambda-\alpha}$, $L_\infty=\mathcal L(\lambda)^-_{\lambda-\alpha}$ and
$$F=\sum_{b\in S}\F \mathtt b_{i1}^{l_1}G_\infty(b)v_\lambda.$$
For $b\in S$, we have $b=G_\infty(b)+q\mathcal L(\infty)$. Assume that $G_\infty(b)$ has the decomposition
$$G_\infty(b)=\sum_{\mathbf c\in \mathscr C_i}\mathtt b_{i,\mathbf c}u_{\mathbf c}\in U^-_{\mathbb A}(\g)_{-\alpha+l_1\alpha_i}\cap \mathcal L(\infty)\cap \mathcal L(\infty)^-.$$
Then we have $$\widetilde f_{i1}\widetilde e_{i1}G_\infty(b)=G_\infty(b)-\sum_{\mathbf c\in \mathscr C_i; 1\notin\mathbf c}\mathtt b_{i,\mathbf c}u_{\mathbf c}=G_\infty(b)-u_b\in q\mathcal L(\infty).$$
Hence we obtain
\begin{itemize}
\item [(i)] $\mathtt b_{i1}^{l_1}G_\infty(b)\equiv \mathtt b_{i1}^{l_1}u_b\ \text{mod}\ (\mathtt b_{i1}^{l_1+1}U^-_q(\g))^{\mathbb A}$, which implies \begin{equation}\label{I1}\mathtt b_{i1}^{l_1}G_\infty(b)v_\lambda\in M\cap (N+L_0)\cap (N+L_\infty).\end{equation}
\item [(ii)] $\widetilde f_{i1}^{l_1}b=\beta_b\mathtt b_{i1}^{l_1}u_b+q\mathcal L(\infty)$ for some $\beta_b\in \F^*$, which implies \begin{equation}\label{I2}\overline{\pi}_\lambda(\widetilde f_{i1}^{l_1}b)=\beta_b\mathtt b_{i1}^{l_1}u_bv_\lambda+q\mathcal L(\lambda).\end{equation}
\end{itemize}

 Set
 $$H:= (N+L_0) /(N+qL_0) \cong \frac{L_0/qL_0}{N \cap L_0 / N\cap qL_0}.$$
 By  induction  hypothesis, we have
 $H\cong \bigoplus_{b\in\mathcal B(\lambda)_{\lambda-\alpha}\backslash (\widetilde f_{i,\mathbf c\cup \{1\}}*\mathcal B(\lambda))}\F b.$ Hence (\ref{S}), (\ref{I1}) and (\ref{I2}) imply that the following canonical  maps are injective:
 \begin{equation}
 \begin{aligned}
 & \varphi_0:F\xrightarrow{\varphi} M\cap(L_0+N)\cap(L_\infty+N)\rightarrow \frac{N+L_0 }{N+qL_0}\xrightarrow{\sim} H \\
 & b_{i1}^{l_1}G_\infty(b)v_\lambda\mapsto\mathtt b_{i1}^{l_1}G_\infty(b)v_\lambda\mapsto
 \mathtt b_{i1}^{l_1}u_bv_\lambda+(N+qL_0) \mapsto \frac{1}{\beta_b}\overline{\pi}_\lambda(\widetilde f_{i1}^{l_1}b)
 \end{aligned}
\end{equation}
 By taking $-$, the following canonical map is injective
 $$\varphi_\infty:F\xrightarrow{\varphi} M\cap(L_0+N)\cap(L_\infty+N)\rightarrow \frac{N+L_\infty }{N+q^{-1}L_\infty}.$$
Note that $M=\mathbb A\varphi(F)+N$.  Hence Lemma \ref{K2} yields
$$M\cap L_0\cap L_\infty\xrightarrow{\sim}  (M\cap L_0) / (M\cap qL_0)  \cong F\oplus  (N\cap L_0/N\cap qL_0),$$
where
$$\begin{aligned}& F\oplus (N\cap L_0/ N\cap qL_0)\cong\bigoplus_{S\sqcup \left(\mathcal B(\lambda)_{\lambda-\alpha}\cap \widetilde f_{i,\mathbf c\cup \{1\}}*\mathcal B(\lambda)\right)}\F b\\
& \phantom{F\oplus (N\cap L_\infty}=\bigoplus_{\mathcal B(\lambda)_{\lambda-\alpha}\cap \left(\widetilde f_{i,\mathbf c}*\mathcal B(\lambda)\right)}\F b.
\end{aligned}$$
This completes the proof.
\end{proof}
}\end{proposition}

\vskip 3mm

\begin{corollary}\label{M1}{\rm
Let $\alpha\in Q_+(r)$.
\begin{itemize}
\item [(i)] For $i\in I^{\text{re}}$ and $n\geq 1$, we have
$$(\mathtt b_i^nU^-_q(\g))^{\mathbb A}_{-\alpha}\cap \mathcal L(\infty)\cap \mathcal L(\infty)^-\xrightarrow{\sim}\frac{(\mathtt b_i^nU^-_q(\g))^{\mathbb A}_{-\alpha}\cap \mathcal L(\infty)}{(\mathtt b_i^nU^-_q(\g))^{\mathbb A}_{-\alpha}\cap q\mathcal L(\infty)}\cong \bigoplus_{b\in \mathcal B(\infty)_{-\alpha}\cap \widetilde f_i^n\mathcal B(\infty)}\F b.$$
\item [(ii)] For $i\in I^{\text{im}}\backslash I^{\text{iso}}$ and $|\mathbf c|\neq 0$, we have
$$(\mathtt b_{i,\mathbf c}U^-_q(\g))^{\mathbb A}_{-\alpha}\cap \mathcal L(\infty)\cap \mathcal L(\infty)^-\xrightarrow{\sim}\frac{(\mathtt b_{i,\mathbf c}U^-_q(\g))^{\mathbb A}_{-\alpha}\cap \mathcal L(\infty)}{(\mathtt b_{i,\mathbf c}U^-_q(\g))^{\mathbb A}_{-\alpha}\cap q\mathcal L(\infty)}\cong \bigoplus_{b\in \mathcal B(\infty)_{-\alpha}\cap \widetilde f_{i,\mathbf c}\mathcal B(\infty)}\F b.$$
\item [(iii)]  For $i\in I^{\text{iso}}$ and any partition  $\mathbf c=(1^{l_1}2^{l_2}\cdots k^{l_k}\cdots)$, we have
$$(\mathtt b_{i,\mathbf c}*U^-_q(\g))^{\mathbb A}_{-\alpha}\cap \mathcal L(\infty)\cap \mathcal L(\infty)^-\xrightarrow{\sim}\frac{(\mathtt b_{i,\mathbf c}*U^-_q(\g))^{\mathbb A}_{-\alpha}\cap \mathcal L(\infty)}{(\mathtt b_{i,\mathbf c}*U^-_q(\g))^{\mathbb A}_{-\alpha}\cap q\mathcal L(\infty)}\cong \bigoplus_{b\in \mathcal B(\infty)_{-\alpha}\cap(\widetilde f_{i,\mathbf c}*\mathcal B(\infty))}\F b,$$
where $\widetilde f_{i,\mathbf c}*\mathcal B(\infty)=\bigcup_{k\geq 1}\widetilde f_{ik}^{l_k}\mathcal B(\infty)$.
\end{itemize}
\begin{proof}
We shall prove (iii) only.  The proof of  (i) and (ii) are similar. For $\lambda\gg 0$, we have
$$U^-_q(\g)_{-\alpha}\xrightarrow{\sim}V(\lambda)_{\lambda-\alpha},\quad (\mathtt b_{i,\mathbf c}*U^-_q(\g))^{\mathbb A}_{-\alpha}\xrightarrow{\sim}(\mathtt b_{i,\mathbf c}*V(\lambda))^{\mathbb A}_{\lambda-\alpha},$$
$$\mathcal L(\infty)_{-\alpha}\xrightarrow{\sim}\mathcal L(\lambda)_{\lambda-\alpha},\quad \mathcal L(\infty)^-_{-\alpha}\xrightarrow{\sim}\mathcal L(\lambda)^-_{\lambda-\alpha},$$
and
$$\bigoplus_{b\in \mathcal B(\infty)_{-\alpha}\cap(\widetilde f_{i,\mathbf c}*\mathcal B(\infty))}\F b\xrightarrow{\sim} \bigoplus_{b\in \mathcal B(\lambda)_{\lambda-\alpha}\cap(\widetilde f_{i,\mathbf c}*\mathcal B(\lambda))}\F b.$$
Hence our assertion follows immediately.
\end{proof}
}\end{corollary}

\vskip 3mm

For $\alpha\in Q_+(r)$ and $(i,l)\in I^{\infty}$, let us denote by $G_{il}$ the inverse of the isomorphism
$$(\mathtt b_{il}U^-_q(\g))^{\mathbb A}_{-\alpha}\cap \mathcal L(\infty)\cap \mathcal L(\infty)^-\xrightarrow{\sim} \bigoplus_{b\in \mathcal B(\infty)_{-\alpha}\cap \widetilde f_{il}\mathcal B(\infty)}\F b.$$
Then Corollary \ref{M1} implies $(\mathtt b_{il}^nU^-_q(\g))^{\mathbb A}_{-\alpha}=\bigoplus_{b\in \mathcal B(\infty)_{-\alpha}\cap \widetilde f_{il}^n\mathcal B(\infty)}\mathbb AG_{il}(b)$ for any $n\geq 1$.

\vskip 3mm

\begin{lemma}{\rm
Let $(i,l),(j,s)\in I^{\infty}$, $\alpha\in Q_+(r)$ and $b\in \widetilde f_{il}\mathcal B(\infty)\cap\widetilde f_{js}\mathcal B(\infty)\cap \mathcal B(\infty)_{-\alpha}$. Then we have
$$G_{il}(b)=G_{js}(b).$$
\begin{proof}
 Let us write $b=\widetilde{f}_{\iota_1}\cdots\widetilde{f}_{kh}\cdot 1$, where $(k,h)\in I^{\infty}$. If $k\in I^{\text{re}}$, then our claim was proved in \cite{Kas91}. So we  will  assume that $k\in I^{\text{im}}$. Take $\lambda\in P^+$ with $\langle h_{k} ,\lambda \rangle =0$ and $\langle h_{j}, \lambda \rangle \gg 0$ for all
$j\in I\backslash\{k\}$. Then (\ref{V}) yields
$$V(\lambda)_{\lambda-\alpha}\simeq U^-_q(\g)_{-\alpha} \bigg/\sum_{n\geq 1}U^-_q(\g)_{-\alpha+n\alpha_k}\mathtt b_{kn}.$$
The same argument in \cite[Lemma 7.5.1]{Kas91} shows that
$$Q=G_{il}(b)-G_{js}(b)\in (\sum_{n\geq 1}U^-_q(\g)_{-\alpha+n\alpha_k}\mathtt b_{kn})\cap U^-_{\mathbb A}(\g)_{-\alpha}\cap q\mathcal L(\infty)\cap\mathcal L(\infty)^-.$$
Then Corollary \ref{C1} implies
$$Q^*\in (\sum_{n\geq 1}\mathtt b_{kn}U^-_q(\g)_{-\alpha+n\alpha_k})\cap U^-_{\mathbb A}(\g)_{-\alpha}\cap q\mathcal L(\infty)\cap\mathcal L(\infty)^-.$$

If $k\in I^{\text{im}}\backslash I^{\text{iso}}$, we assume that $Q^*=\mathtt b_{k1}u_1+\cdots+\mathtt b_{ks}u_s$. Since $\overline{Q^*}=\overline Q^*=Q^*$, we have $Q^*=\mathtt b_{k1}\overline u_1+\cdots+\mathtt b_{ks}\overline u_s$. Note that for each $1\leq j\leq s$, $\mathtt b_{kj}u_j=\widetilde f_{kj}\widetilde e_{kj}Q^*\in \mathtt b_{kj}U^-_q(\g)\cap U^-_{\mathbb A}(\g)_{-\alpha}\cap q\mathcal L(\infty)$ and $u_j=\overline u_j=\widetilde e_{kj}Q^*$. Hence
$b_{kj}u_j\in (\mathtt b_{kj}U^-_q(\g))^{\mathbb A}_{-\alpha}\cap q\mathcal L(\infty)\cap \mathcal L(\infty)^-$ and Corollary \ref{M1} (ii) implies $Q^*=0$.

If $k\in I^{\text{iso}}$, since $Q^*\in (\sum_{n\geq 1}\mathtt b_{kn} U^-_q(\g)_{-\alpha+n\alpha_k})\cap U^-_{\mathbb A}(\g)_{-\alpha}$, the decomposition of
$Q^*$ can be expressed as the form $Q^*=\mathtt b_{k1}u_1+\cdots+\mathtt b_{ks}u_s$ with $$u_j=\sum_{\substack{\mathbf c\in \mathscr C_k\ \text{and}  \\ \mathbf c\ \text{contains no} \ j+1,\cdots, s}}\mathtt b_{k,\mathbf c}u_{\mathbf c}.$$
 For every $1\leq j\leq s$, we have $$\mathtt b_{kj}u_{j}=\widetilde f_{kj}\widetilde e_{kj}(Q^*-\sum_{j<p\leq s}b_{kp}u_p)\in (\mathtt b_{kj}U^-_q(\g))^{\mathbb A}_{-\alpha}.$$
Hence $Q^*\in (\mathtt b_{k,(1,\dots, s)}*U^-_q(\g))^{\mathbb A}_{-\alpha}$, and Corollary \ref{M1} (iii) implies $Q^*=0$.
\end{proof}
}\end{lemma}

%{\color{red} (Bolun, as I wrote in my note, my proof is a little bit different from yours. For example, I used $\mathbf{c}
%=(1,1, \ldots, 1)$ instead of $\mathbf{c} = (1, 2, \ldots, s)$. I will just trust your proof. Please compare your proof with mine.)}

\vskip 3mm

Thus we can define
$$G:\mathcal L(\infty)_{-\alpha}/q\mathcal L(\infty)_{-\alpha}\rightarrow U^-_{\mathbb A}(\g)_{-\alpha}\cap \mathcal L(\infty)\cap \mathcal L(\infty)^-$$
by
$$G(b)\longmapsto G_{il}(b) \ \ \text{for} \  b\in \widetilde f_{il}\mathcal B(\infty)\cap \mathcal B(\infty)_{-\alpha},  (i,l) \in I^{\infty}.$$
Then we have $b=G(b)+q\mathcal L(\infty)$ and
\begin{equation}\label{G(b)}
(\mathtt b_{il}^nU^-_q(\g))^{\mathbb A}_{-\alpha}=\bigoplus_{b\in \mathcal B(\infty)_{-\alpha}\cap \widetilde f_{il}^n\mathcal B(\infty)}\mathbb AG(b)
\end{equation}
for any $n\geq 1$. Since $U^-_{\mathbb A}(\g)_{-\alpha}=\sum_{(i,l)\in I^{\infty}}(\mathtt b_{il}U^-_q(\g))^{\mathbb A}_{-\alpha}$, we obtain $$U^-_{\mathbb A}(\g)_{-\alpha}=\sum_{b\in \mathcal B(\infty)_{-\alpha}}\mathbb A G(b).$$
Let $E=\mathcal L(\infty)_{-\alpha}/q\mathcal L(\infty)_{-\alpha}$ and $M=U^-_{\mathbb A}(\g)_{-\alpha}$.  Then by Lemma \ref{K1} (ii), we deduce that
$$\mathcal L(\infty)_{-\alpha}/q\mathcal L(\infty)_{-\alpha}\xrightarrow{G} U^-_{\mathbb A}(\g)_{-\alpha}\cap \mathcal L(\infty)\cap \mathcal L(\infty)^-\rightarrow \frac{U^-_{\mathbb A}(\g)_{-\alpha}\cap \mathcal L(\infty)}{U^-_{\mathbb A}(\g)_{-\alpha}\cap q\mathcal L(\infty)}$$
is an isomorphism, which proves A($r$). Now C($r$) follows from (\ref{G(b)}). Finally, we shall prove B($r$).

\vskip 3mm

\begin{lemma}{\rm
Let $\alpha\in Q_+(r)$, $b\in \mathcal B(\infty)_{-\alpha}$ and $\lambda\in P^+$. If $\overline{\pi}_\lambda (b)=0$, then $G(b)v_\lambda=0$.
\begin{proof}
Take $(i,l)\in I^{\infty}$ with $\widetilde e_{il}b\neq 0$. Then $G(b)v_\lambda\in (\mathtt b_{il}V(\lambda))^{\mathbb A}_{\lambda-\alpha}\cap q\mathcal L(\lambda)\cap \mathcal L(\lambda)^-$ by Proposition \ref{M}.
\end{proof}
}\end{lemma}

\vskip 3mm

By this lemma, we have
$V(\lambda)^{\mathbb A}_{\lambda-\alpha}=\sum_{\substack{b\in \mathcal B(\infty)_{-\alpha}\\ \overline{\pi}_\lambda(b)\neq 0}}\mathbb A G(b)v_\lambda$.
Let $$E=\sum_{\substack{b\in \mathcal B(\infty)_{-\alpha}\\ \overline{\pi}_\lambda(b)\neq 0}}\F G(b)v_\lambda\subseteq V(\lambda)^{\mathbb A}_{\lambda-\alpha}\cap \mathcal L(\lambda)\cap\mathcal L(\lambda)^-.$$
Since $\{b\in\mathcal B(\infty)_{-\alpha}\mid \overline{\pi}_\lambda(b)\neq 0\}\xrightarrow{\sim}\mathcal B(\lambda)_{\lambda-\alpha}$, we have
$$E\xrightarrow{\sim} \mathcal L(\lambda)_{\lambda-\alpha}/q\mathcal L(\lambda)_{\lambda-\alpha}$$ given by
$$G(b)v_\lambda\longmapsto G(b)v_\lambda+q\mathcal L(\lambda)=\overline{\pi}_\lambda(b).$$
By Lemma \ref{K1} (ii), we get
$$E\xrightarrow{\sim} V(\lambda)^{\mathbb A}_{\lambda-\alpha}\cap \mathcal L(\lambda)\cap\mathcal L(\lambda)^- \xrightarrow{\sim} \frac{V(\lambda)^{\mathbb A}_{\lambda-\alpha}\cap \mathcal L(\lambda)}{V(\lambda)^{\mathbb A}_{\lambda-\alpha}\cap q\mathcal L(\lambda)}\cong \mathcal L(\lambda)_{\lambda-\alpha}/q\mathcal L(\lambda)_{\lambda-\alpha},$$
which proves B($r$).

\vskip 3mm

To summarize, we  obtain the main goal of this paper.

\vskip 3mm

\begin{theorem}\label{A}{\rm
There exist canonical isomorphisms
\begin{equation*}
\begin{aligned}
U^-_{\mathbb A}(\g)\cap \mathcal L(\infty)\cap\mathcal L(\infty)^-
& \xrightarrow{\sim} \mathcal L(\infty)/q\mathcal L(\infty),\\
V(\lambda)^{\mathbb A}\cap \mathcal L(\lambda)\cap\mathcal L(\lambda)^-
& \xrightarrow{\sim} \mathcal L(\lambda)/q\mathcal L(\lambda) \ \ \ (\lambda \in P^{+}).
\end{aligned}
\end{equation*}
}
\end{theorem}

\vskip 3mm

\begin{definition} \hfill

\vskip 2mm

{\rm

(a) ${\mathbf B} := \{ G_\infty(b) \mid b \in B(\infty) \}$ is called the {\it global basis} of $U_{\mathbb A}^{-}(\g)$
corresponding to $B(\infty)$.

\vskip 2mm

(b) ${\mathbf B}^{\lambda} := \{ G_{\lambda}(b) \mid b \in B(\lambda) \}$ is called the {\it global basis} of $V(\lambda)^{\mathbb A}$ corresponding to $B(\lambda)$.

}
\end{definition}

\vskip 3mm

\begin{remark}

The global bases ${\mathbf B}$ and ${\mathbf B}^{\lambda}$ are unique because they are stable under the bar involution.

\end{remark}
%\clearpage

\end{document}